\documentclass[12pt]{amsart}

\usepackage{amssymb}
\usepackage{bm}
\usepackage{graphicx}
\usepackage{psfrag}
\usepackage{color}
\usepackage{hyperref}
\hypersetup{colorlinks=true, linkcolor=blue, citecolor=cyan, urlcolor=wine}
\usepackage{url}
\usepackage{algpseudocode}
\usepackage{fancyhdr}
\usepackage{xy}
\input xy
\xyoption{all}
\usepackage{stmaryrd}

\newtheorem*{maintheorem*}{Main Theorem}
\newtheorem{theorem}{Theorem}[section]
\newtheorem{prop}[theorem]{Proposition}
\newtheorem{question}[theorem]{Question}

\newtheorem{lemma}[theorem]{Lemma}
\newtheorem{cor}[theorem]{Corollary}

\theoremstyle{definition}
\newtheorem{definition}[theorem]{Definition}
\newtheorem{remark}[theorem]{Remark}
\newtheorem{example}[theorem]{Example}
\numberwithin{equation}{section}

\newcommand{\cc}{\mathbb{C}}
\newcommand{\ff}{\mathbb{F}}
\newcommand{\nn}{\mathbb{N}}
\newcommand{\qq}{\mathbb{Q}}
\newcommand{\rr}{\mathbb{R}}
\newcommand{\zz}{\mathbb{Z}}
\providecommand\ldb{\llbracket}
\providecommand\rdb{\rrbracket}

\newcommand{\slp}{\text{slope}}
\newcommand{\conv}{\mathsf{conv}}
\newcommand{\cone}{\mathsf{cone}}
\newcommand{\aff}{\mathsf{aff}}
\newcommand{\lin}{\mathsf{lin}}

\newcommand{\gp}{\text{gp}}
\newcommand{\inter}{\mathsf{int}}
\newcommand{\bd}{\mathsf{bd}}
\newcommand{\rank}{\text{rank}}
\newcommand{\rk}{\text{rank}}

\newcommand{\relin}{\mathsf{relin}}

\newcommand{\pp}{\mathsf{P}}

\newcommand{\norm}[1]{\left\lVert#1\right\rVert}



\begin{document}
	
	\mbox{}
	\title{Geometric and combinatorial aspects of submonoids \\ of a finite-rank free commutative monoid}
	
%	\author{Felix Gotti}
%	\address{Department of Mathematics\\UC Berkeley\\Berkeley, CA 94720}
%	\email{felixgotti@berkeley.edu}
	
	\author{Felix Gotti}
	\address{Department of Mathematics\\MIT\\Cambridge, MA 02139}
	\email{fgotti@mit.edu}

	\date{\today}
	
	\begin{abstract}
		If $\ff$ is an ordered field and $M$ is a finite-rank torsion-free monoid, then one can embed $M$ into a finite-dimensional vector space over $\ff$ via the inclusion $M \hookrightarrow \gp(M) \hookrightarrow \ff \otimes_{\zz} \gp(M)$, where $\gp(M)$ is the Grothendieck group of $M$. Let~$\mathcal{C}$ be the class consisting of all monoids (up to isomorphism) that can be embedded into a finite-rank free commutative monoid. Here we investigate how the atomic structure and arithmetic properties of a monoid $M$ in $\mathcal{C}$ are connected to the combinatorics and geometry of its conic hull $\cone(M) \subseteq \ff \otimes_{\zz} \gp(M)$. First, we show that the submonoids of $M$ determined by the faces of $\cone(M)$ account for all divisor-closed submonoids of $M$. Then we appeal to the geometry of $\cone(M)$ to characterize whether $M$ is a factorial, half-factorial, and other-half-factorial monoid. Finally, we investigate the cones of finitary, primary, finitely primary, and strongly primary monoids in $\mathcal{C}$. Along the way, we determine the cones that can be realized by monoids in~$\mathcal{C}$ and by finitary monoids in $\mathcal{C}$.
	\end{abstract}

\maketitle

\tableofcontents

\medskip
%%%%%%%%%%%%%
%%%%%%%%%%%%%
\section{Introduction}
\label{sec:intro 6}

Let $\mathcal{C}$ denote the class containing, up to isomorphism, each monoid that can be embedded into a finite-rank free commutative monoid. If $\ff$ is an ordered field and $M$ is a monoid in $\mathcal{C}$, then the chain of natural inclusions
\[
	M \hookrightarrow \gp(M) \hookrightarrow V := \ff \otimes_\zz \gp(M)
\]
yields an embedding of $M$ into the finite-dimensional $\ff$-vector space $V$, where $\gp(M)$ is the Grothendieck group of $M$. Many connections between the algebraic properties of a finitely generated monoid $M$ in $\mathcal{C}$ and the geometric properties of its cone $\cone_V(M)$ in $V$ have been established in the past; see, for instance,~\cite[Propositions~2.16, 2.36, 2.37, and~2.40]{BG09}. However, it seems that the search for similar connections for the non-finitely generated monoids in $\mathcal{C}$ has only received isolated attention. In this paper we bring to the reader an initial systematic study of the connection between both arithmetic and atomic aspects of monoids $M$ in $\mathcal{C}$ and both the geometry of $\cone_V(M)$ and the combinatorics of the face lattice of $\cone_V(M)$. Here we primarily focus on the non-finitely generated members of $\mathcal{C}$.

As we proceed to illustrate, various subclasses of $\mathcal{C}$ consisting of non-finitely generated monoids have appeared in the literature of commutative algebra and algebraic geometry. For example, semigroups of values of singular curves were first studied by Delgado in~\cite{fDM87}, and then generalized by Barucci, D'Anna, and Fr\"oberg in~\cite{BDF00} under the term ``good semigroups". Good semigroups are monoids in $\mathcal{C}$ that are not, in general, finitely generated (see Example~\ref{ex:good semigroups}). On the other hand, for a cofinite submonoid $\Gamma$ of $(\nn,+)$ and $s \in \Gamma$, let $M^s_{\Gamma}$ be the monoid consisting of all arithmetic progressions of step size $s$ that are contained in $\Gamma$. Monoids parameterized by the set of pairs $(\Gamma, s)$ are members of $\mathcal{C}$ known as Leamer monoids. The atomic structure of Leamer monoids is connected to the Huneke-Wiegand Conjecture (see Example~\ref{ex:Leamer monoids} for more details). Finally, let $\alpha$ and $\beta$ be two positive irrational numbers such that $\alpha < \beta$, and consider the submonoid $M_{\alpha, \beta}$ of $(\nn^2,+)$ that results from intersecting $\zz^2$ and the cone in the first quadrant of $\rr^2$ determined by the lines $y = \alpha x$ and $y = \beta x$. The monoids $M_{\alpha, \beta}$, which clearly belong to $\mathcal{C}$, show up in the study of Betti tables of short Gorenstein algebras (see Example~\ref{ex:irrational cones}).

It follows immediately that the conic hulls of the monoids $M^s_\Gamma$ and $M_{\alpha, \beta}$ are not polyhedral cones in $\rr^2$ because they are not closed with respect to the Euclidean topology. Then Farkas-Minkowski-Weyl Theorem ensures that the monoids $M^s_\Gamma$ and $M_{\alpha, \beta}$ cannot be finitely generated. Just as we did right now, we shall be using Farkas-Minkowski-Weyl Theorem systematically throughout this paper since it is the crucial tool guaranteeing that most of the monoids in $\mathcal{C}$ we will be focused on cannot be finitely generated.

A cancellative commutative monoid is called atomic if each non-invertible element factors into irreducibles. Note that all monoids in $\mathcal{C}$ are atomic. As for integral domains, a monoid is called a UFM (or a unique factorization monoid) provided that each non-invertible element has an essentially unique factorization into irreducibles. Any UFM is clearly atomic. A huge variety of atomic conditions between being atomic and being a UFM have been considered in the literature, including half-factoriality, other-half-factoriality, being finitary, and being strongly primary. In this paper, we investigate all the mentioned intermediate atomic conditions for members of~$\mathcal{C}$ in terms of the cones they generate.

Let $M$ be an atomic monoid. Then $M$ is called an HFM (or a half-factorial monoid) if for each non-invertible $x \in M$, any two factorizations of $x$ have the same number of irreducibles (counting repetitions). On the other hand, $M$ is called an OHFM (or an other-half-factorial monoid) if for each non-invertible $x \in M$ no two essentially distinct factorizations of~$x$ contain the same number of irreducibles (counting repetitions). Also,~$M$ is called primary if it is nontrivial and for all nonzero non-invertible $x,y \in M$ there exists $n \in \nn$ such that $ny \in x + M$. If $M$ belongs to $\mathcal{C}$, then $M$ is called finitary if there exist a finite subset $S$ of $M$ and a positive integer $n$ such that $n (M \! \setminus \! \{0\}) \subseteq S + M$.

The rest of this paper is structured as follows. In Section~\ref{sec:background} we review the main concepts in commutative monoids and convex geometry we will need in later sections. In Section~\ref{sec:cones realized} we present some preliminary results about monoids in $\mathcal{C}$ and their cones. The main result in this section is Theorem~\ref{thm:a characterization of cones generated by monoids in C}, where we determine the cones generated by monoids in $\mathcal{C}$. In Section~\ref{sec:divisor-closed submonoids and finitely primary monoids} we study the submonoids of a monoid $M$ in $\mathcal{C}$ determined by the faces of $\cone_V(M)$, which we call face submonoids. The most important result we achieve in this section is Theorem~\ref{thm:characterization of divisor-closed submonoids in $C$}, where face submonoids are characterized; such a characterization allows us to deduce a characterization of the cones of primary monoids in $\mathcal{C}$, which was previously obtained in~\cite{GHL95}. Then in Section~\ref{sec:factoriality} we offer geometric characterizations for the UFMs, HFMs, and OHFMs of~$\mathcal{C}$; we do this in Theorems~\ref{thm:factoriality characterizations},~\ref{thm:HF characterizations}, and~\ref{thm:OHF characterizations}, respectively. Lastly, Section~\ref{sec:primary monoids and finitary monoids} is devoted to study the cones of primary monoids and finitary monoids. The main results in this section are Theorem~\ref{thm:the closure of the cone of a finitely primary monoid is rational and simplicial}, which states that if a monoid $M$ in $\mathcal{C}$ is finitely primary, then the closure of the cone it generates (when $\ff = \rr$) is a rational simplicial cone, and Theorem~\ref{thm:finitary GAMs}, which states that a monoid $M$ in $\mathcal{C}$ is finitary provided that the cone it generates is polyhedral.

\medskip
%%%%%%%%%%%%%%%%%%%%%%%%%%
%%%%%%%%%%%%%%%%%%%%%%%%%%
\section{Atomic Monoids and Convex Cones}
\label{sec:background}

In this section we introduce most of the relevant concepts related to commutative semigroups and convex geometry required to follow the results presented later. General references for any undefined terminology or notation can be found in~\cite{pG01} for commutative semigroups, in~\cite{GH06b} for atomic monoids, and in~\cite{rtR70} for convex geometry.

\smallskip
%%%%%%%%%%%%%%%%%%
\subsection{General Notation}

Set $\nn := \{0,1,2,\dots\}$. If $x,y \in \zz$, then we let $\ldb x,y \rdb$ denote the interval of integers between $x$ and $y$, i.e., 
\[
	\ldb x,y \rdb := \{z \in \zz \mid x \le z \le y\}.
\]
Clearly, $\ldb x,y \rdb$ is empty when $x > y$. In addition, for $X \subseteq \rr$ and $r \in \rr$, we set
\[
	X_{\ge r} := \{x \in X \mid x \ge r\},
\]
and we use the notation $X_{> r}$ in a similar way. Lastly, if $Y \subseteq \rr^d$ for some $d \in \nn \setminus \{0\}$, then we set $Y^\bullet := Y \setminus \{0\}$.

\smallskip
%%%%%%%%%%%%%%%%%
\subsection{Atomic Monoids}

A \emph{monoid} is commonly defined in the literature as a semigroup along with an identity element. However, we shall tacitly assume that all monoids here are also commutative and cancellative, omitting these two attributes accordingly. As we only consider commutative monoids, unless otherwise specified we will use additive notation. In particular, the identity element of a monoid $M$ will be denoted by $0$, and we let $M^\bullet$ denote the set $M \! \setminus \! \{0\}$. A monoid is called \emph{reduced} if its only invertible element is the identity element. Unless we state otherwise, monoids here are also assumed to be reduced.

Let $M$ be a monoid. We write $M = \langle S \rangle$ when $M$ is generated by a set $S$, i.e., $S \subseteq M$ and no proper submonoid of $M$ contains $S$. If $M$ can be generated by a finite set, then~$M$ is said to be \emph{finitely generated}. An element $a \in M^\bullet$ is called an \emph{atom} if for each pair of elements $x,y \in M$ the equation $a = x+y$ implies that either $x = 0$ or $y = 0$. The set consisting of all atoms of $M$ is denoted by $\mathcal{A}(M)$. It immediately follows that
\[
	\mathcal{A}(M) := M^\bullet \setminus \big( M^\bullet + M^\bullet \big).
\]
Since $M$ is reduced, $\mathcal{A}(M)$ will be contained in each generating set of $M$. However, $\mathcal{A}(M)$ may not generate $M$: for instance, taking $M$ to be the submonoid of $(\qq_{\ge 0},+)$ generated by $\{1/2^n \mid n \in \nn\}$ one can see that $\mathcal{A}(M) = \emptyset$ and, therefore, $\mathcal{A}(M)$ does not generate $M$. If $\mathcal{A}(M)$ generates $M$, then $M$ is said to be \emph{atomic}. From now on, all monoids addressed in this paper are assumed to be atomic. For $x,y \in M$, we say that $x$ \emph{divides} $y$ \emph{in} $M$ and write $x \mid_M y$ provided that $y = x + x'$ for some $x' \in M$. An element $p \in M^\bullet$ is said to be \emph{prime} if whenever $p \mid_M x + y$ for some $x,y \in M$ either $p \mid_M x$ or $p \mid_M y$. The monoid $M$ is called a \emph{UFM} (or a \emph{unique factorization monoid}) if each nonzero element can be written as a sum of primes (in which case such a sum is unique up to permutation). Clearly, every prime element of $M$ is an atom. Thus, if~$M$ is a UFM, then it is, in particular, an atomic monoid.

A subset $I$ of $M$ is called an \emph{ideal} if $I + M = I$ (or, equivalently, $I + M \subseteq I$). An ideal $I$ is \emph{principal} if $I = x + M$ for some $x \in M$. Furthermore, $M$ satisfies the \emph{ACCP} (or the \emph{ascending chain condition on principal ideals}) provided that each increasing sequence of principal ideals of~$M$ eventually stabilizes. It is well known that each monoid satisfying the ACCP is atomic \cite[Proposition~1.1.4]{GH06b}. Gram's monoid, introduced in~\cite{aG74}, is an atomic monoid that does not satisfy the ACCP.

For the monoid $M$ there exist an abelian group $\text{gp}(M)$ and a monoid homomorphism $\iota \colon M \hookrightarrow \text{gp}(M)$ such that any monoid homomorphism $\phi \colon M \to G$, where $G$ is an abelian group, uniquely factors through $\iota$. The group $\text{gp}(M)$, which is unique up to isomorphism, is called the \emph{Grothendieck group}\footnote{The Grothendieck group of a monoid is often called the difference or the quotient group depending on whether the monoid is written additively or multiplicatively.} of $M$. The monoid $M$ is \emph{torsion-free} if for all $n \in \nn^\bullet$ and $x,y \in M$ the equality $nx = ny$ implies that $x = y$. A monoid is torsion-free if and only if its Grothendieck group is torsion-free (see \cite[Section~2.A]{BG09}). If~$M$ is torsion-free, then the \emph{rank} of~$M$, denoted by $\rk(M)$, is the rank of the $\zz$-module $\text{gp}(M)$, that is, the dimension of the $\qq$-space $\qq \otimes_\zz \text{gp}(M)$.

A multiplicative monoid $F$ is said to be \emph{free on} a subset $A$ of $F$ provided that each element $x \in F$ can be written uniquely in the form
\[
	x = \prod_{a \in A} a^{\mathsf{v}_a(x)},
\]
where $\mathsf{v}_a(x) \in \nn$ and $\mathsf{v}_a(x) > 0$ only for finitely many $a \in A$. It is well known that for each set $A$, there exists a unique (up to isomorphism) monoid $F$ such that $F$ is a free (commutative) monoid on $A$. For the monoid $M$, we let $\mathsf{Z}(M)$ denoted the free (commutative) monoid on $\mathcal{A}(M)$, and we call the elements of $\mathsf{Z}(M)$ \emph{factorizations}. If $z = a_1 \cdots a_n$ is a factorization in $\mathsf{Z}(M)$ for some $n \in \nn$ and $a_1, \dots, a_n \in \mathcal{A}(M)$, then~$n$ is called the \emph{length} of $z$ and is denoted by $|z|$. There is a unique monoid homomorphism $\pi \colon \mathsf{Z}(M) \to M$ satisfying that $\pi(a) = a$ for all $a \in \mathcal{A}(M)$. %is called the \emph{factorization homomorphism} of $M$. 
For each $x \in M$ the set
\[
	\mathsf{Z}(x) := \mathsf{Z}_M(x) := \pi^{-1}(x) \subseteq \mathsf{Z}(M)
\]
is called the \emph{set of factorizations} of $x$. In addition, for $k \in \nn$ we set
\[
	\mathsf{Z}_k(x) := \{z \in \mathsf{Z}(x) : |z| = k \} \subseteq \mathsf{Z}(M).
\]
Observe that $M$ is atomic if and only if $\mathsf{Z}(x)$ is nonempty for all $x \in M$ (notice that $\mathsf{Z}(0) = \{\emptyset\}$). By~\cite[Theorem~1.2.9]{GH06b}, $M$ is a UFM if and only if $|\mathsf{Z}(x)| = 1$ for all $x \in M$. The monoid $M$ is called an \emph{FFM} (or a \emph{finite factorization monoid}) if $|\mathsf{Z}(x)| < \infty$ for all $x \in M$. For each $x \in M$, the \emph{set of lengths} of $x$ is defined by
\[
	\mathsf{L}(x) := \mathsf{L}_M(x) := \{|z| :  z \in \mathsf{Z}(x)\}.
\]
The sets of lengths of submonoids of $(\nn^d,+)$, where $d \in \nn$, have been considered in~\cite{fG19c}. If $|\mathsf{L}(x)| < \infty$ for all $x \in M$, then $M$ is called a \emph{BFM} (or a \emph{bounded factorization monoid}). Clearly, each FFM is a BFM. In addition, each BFM satisfies the ACCP (see \cite[Corollary~1.3.3]{GH06b}).

The monoid $M$ is called \emph{finitary} if it is a BFM and there exist a finite subset $S$ of $M$ and $n \in \nn^\bullet$ such that $n M^\bullet \subseteq S + M$. Clearly, finitely generated monoids are finitary. On the other hand, $M$ is called \emph{primary} provided that $M$ is nontrivial and for all $x,y \in M^\bullet$ there exists $n \in \nn$ such that $ny \in x + M$. If $M$ is both finitary and primary, then it is called \emph{strongly primary}.
 
A \emph{numerical monoid} is a cofinite submonoid of $(\nn,+)$. Each numerical monoid is finitely generated and, therefore, atomic. In addition, it is not hard to see that numerical monoids are strongly primary. An introduction to numerical monoids can be found in~\cite{GR09}. The class of finitely generated submonoids of $(\nn^d,+)$ naturally generalizes that one of numerical monoids. Although members of the former class are obviously finitary, they need not be strongly primary; indeed, numerical monoids are the only primary monoids in this class (see Proposition~\ref{prop:primary characterization}). However, we shall see later that there are many non-finitely generated submonoids of $(\nn^d,+)$ that are primary. In addition, we will provide necessary and sufficient conditions for a submonoid of $(\nn^d,+)$ to be finitary.

\smallskip
%%%%%%%%%%%%%%%%
\subsection{Convex Cones}

For the rest of this section, fix $d \in \nn^\bullet$. As usual, for $x \in \rr^d$ we let $\norm{x}$ denote the Euclidean norm of~$x$. We always consider the space $\rr^d$ endowed with the topology induced by the Euclidean norm. In addition, if $\ff$ is an ordered field such that $\qq \subseteq \ff \subseteq \rr$, we let the vector space $\ff^d$ inherit the topology of~$\rr^d$. For a subset $S$ of~$\rr^d$, we let $\inter \, S$, $\bar{S}$, and $\bd \, S$ denote the interior, closure, and boundary of~$S$, respectively.

Let $\ff$ be an ordered field, and let $V$ be a $d$-dimensional vector space over $\ff$ for some $d \in \nn^\bullet$. When $V = \ff^d$, we denote the standard inner product of $\ff^d$ by $\langle \, , \rangle$, that is, $\langle x, y \rangle = \sum_{i=1}^d x_i y_i$ for all $x = (x_1, \dots, x_d)$ and $y = (y_1, \dots, y_d)$ in $\ff^d$. The convex hull of $X \subseteq V$ (i.e., the intersection of all convex subsets of $V$ containing $X$) is denoted by $\conv_V(X)$. A nonempty convex subset $C$ of $V$ is called a \emph{cone} if $C$ is closed under linear combinations with nonnegative coefficients. A cone $C$ is called \emph{pointed} if $C \cap -C = \{0\}$. Unless otherwise stated, we assume that the cones we consider here are pointed. If $X$ is a nonempty subset of $V$, then the \emph{conic hull} of $X$, denoted by $\cone_V(X)$, is defined as follows:
\[
	\cone_V(X) := \bigg\{ \sum_{i=1}^n c_i x_i \ \bigg{|} \ n \in \nn, \ \text{and} \ x_i \in X \ \text{and} \ c_i \ge 0 \ \text{for every} \ i \in \ldb 1,n \rdb \bigg\},
\]
i.e., $\cone_V(X)$ is the smallest cone in $V$ containing $X$. When there is no risk of ambiguity, we write $\cone(X)$ instead of $\cone_V(X)$. A cone in $V$ is called \emph{simplicial} if it is the conic hull of a linearly independent set of vectors.

Let $C$ be a cone in $V$. A \emph{face} of~$C$ is a cone $F$ contained in $C$ and satisfying the following condition: for all $x,y \in C$ the fact that $F$ intersects the open line segment $\{tx + (1-t)y \mid t \in \ff \text{ and } 0 < t < 1\}$ implies that both $x$ and $y$ belong to $F$. If $F$ is a face of $C$ and $F'$ is a face of $F$, then it is clear that $F'$ must be a face of $C$. For a nonzero vector $u \in V$, consider the hyperplane $H := \{x \in V \mid \langle x, u \rangle = 0 \}$, and denote the closed half-spaces $\{x \in V \mid \langle x, u \rangle \le 0 \}$ and $\{x \in V \mid \langle x, u \rangle \ge 0 \}$ by $H^-$ and $H^+$, respectively. If a cone~$C$ satisfies that $C \subseteq H^-$ (resp., $C \subseteq H^+$), then $H$ is called a \emph{supporting hyperplane} of~$C$ and~$H^-$ (resp., $H^+$) is called a \emph{supporting half-space} of $C$. A face $F$ of $C$ is called \emph{exposed} if there exists a supporting hyperplane $H$ of $C$ such that $F = C \cap H$.

The cone $C$ is called \emph{polyhedral} provided that it can be expressed as the intersection of finitely many closed half-spaces. Farkas-Minkowski-Weyl Theorem states that a convex cone is polyhedral if and only if it is the conic hull of a finite set~\cite[Section~1]{CC58}. It is clear that every simplicial cone is polyhedral. When $V = \ff^d$, a cone in~$V$ is called \emph{rational} if it is the conic hull of vectors with integer coordinates. Gordan's Lemma states that if $C$ is a rational polyhedral cone in $\rr^d$ and $G$ is a subgroup of $(\qq^d,+)$, then $C \cap G$ is finitely generated~\cite[Lemma 2.9]{BG09}.

A subset $A$ of $V$ is called an \emph{affine set} (or an \emph{affine subspace}) provided that for all $x,y \in A$ with $x \neq y$, the line determined by $x$ and $y$ is contained in $A$. Affine sets are translations of subspaces, and an $(n-1)$-dimensional affine set is called an \emph{affine hyperplane}. The \emph{affine hull} of a subset $S$ of $V$, denoted by $\aff_V(S)$, is the smallest affine set containing~$S$. In addition, we let $\lin_V(S)$ denote the smallest subspace of $V$ containing $S$. Clearly, $\lin_V(S) = \aff_V(\cone_V(S))$. Let $V'$ be also a finite-dimensional vector space over~$\ff$, and let $C$ and~$C'$ be cones in $V$ and~$V'$, respectively. Then $C$ and~$C'$ are called \emph{affinely equivalent} if there exists an invertible linear transformation from $\aff_V(C)$ to $\aff_{V'}(C')$ mapping $C$ onto~$C'$. When $\ff$ is an intermediate field of the extension $\qq \subseteq \rr$, the relative interior of $S$, denoted by $\relin(S)$, is the Euclidean interior of $S$ when considered as a subset of $\aff_V(S)$. If $C$ is a cone in $\rr^d$, then $C$ is the disjoint union of all the relative interiors of its nonempty faces~\cite[Theorem~18.2]{rtR70}.

\medskip
%%%%%%%%%%%%%%%%%%%%%%%%%%%%%%%
%%%%%%%%%%%%%%%%%%%%%%%%%%%%%%%
\section{Monoids in $\mathcal{C}$ and Their Cones}
\label{sec:monoids in C}

%\smallskip
%%%%%%%%%%%%%%%%%%%%%%%%%%%%%
\subsection{Monoids in the Class $\mathcal{C}$}

In this section we introduce the class of monoids we shall be concerned with throughout this paper. We also introduce the cones associated to such monoids.

\begin{prop} \label{prop:conditions defining monoids in C}
	For a monoid $M$, the following statements are equivalent.
	\begin{enumerate}
		\item[(a)] The monoid $M$ can be embedded into a finite-rank free (commutative) monoid.
		\smallskip
		
		\item[(b)] The monoid $M$ has finite rank and can be embedded into a free (commutative) monoid.
		\smallskip
		
		\item[(c)] There exists $d \in \nn$ such that $M$ can be embedded in $(\nn^d,+)$ as a maximal-rank submonoid.
	\end{enumerate}
\end{prop}

\begin{proof}
	(a) $\Rightarrow$ (b): Suppose that $F$ is a free (commutative) monoid of finite rank containing $M$, and assume that $\gp(M) \subseteq \gp(F)$. Then one can consider $\gp(M)$ as a $\zz$-submodule of $\gp(F)$. Since $\gp(F)$ is a finite-rank $\zz$-module, so is $\gp(M)$. Hence~$M$ has finite rank, which yields~(b).
	\smallskip
	
	(b) $\Rightarrow$ (a): To argue this implication, let $X$ be a set such that $M$ is embedded into the free (commutative) monoid $\bigoplus_{x \in X} \nn x$. After identifying $M$ with its image, we can assume that $M \subseteq \bigoplus_{x \in X} \nn x$ and identify $\gp(M)$ with a subgroup of $\bigoplus_{x \in X} \zz x$ containing $M$. Since $\rk(M) < \infty$, the dimension of the subspace $W$ of $V := \bigoplus_{x \in X} \qq x$ generated by $\gp(M)$ is finite. Let $\{b_1, \dots, b_k\}$ be a basis of $W$ for some $k \in \nn$. For each $i \in \ldb 1,k \rdb$ there exists a finite subset $Y_i$ of $X$ such that $b_i \in \bigoplus_{x \in Y_i} \qq x$. As a result, $W \subseteq \bigoplus_{y \in Y} \qq y$, where $Y = \bigcup_{i=1}^k Y_i$. Then one finds that
	\[
		M \subseteq \bigg( \bigoplus_{y \in Y} \qq y \bigg) \bigcap \bigg( \bigoplus_{x \in X} \nn x \bigg) = \bigoplus_{y \in Y} \nn y.
	\]
	Because $|Y| < \infty$, the monoid $\bigoplus_{y \in Y} \nn y$ is a free (commutative) monoid of finite rank containing $M$, and so~(1) holds.
	\smallskip
	
	(c) $\Rightarrow$ (a): It follows trivially.
	 \smallskip
	 
	(a) $\Rightarrow$ (c): To prove this last implication, let~$M$ be a monoid of rank $d$, and suppose that $M$ is a submonoid of a free (commutative) monoid of rank~$r$ for some $r \in \nn_{\ge d}$. There is no loss of generality in assuming that $M$ is a submonoid of $(\nn^r,+)$. Let~$V$ be the subspace of $\qq^r$ generated by~$M$. Since $\rk(M) = d$, the subspace $V$ has dimension~$d$. Now consider the submonoid $M' := \nn^r \cap V$ of $(\nn^r,+)$. As $M'$ is the intersection of the rational cone $\cone(\nn^r \cap V)$ and the lattice $\zz^r \cap V \cong (\zz^d,+)$, it follows from Gordan's Lemma that $M'$ is finitely generated. On the other hand, $M \subseteq M' \subseteq V$ guarantees that $\rk(M') = d$. Since $M'$ is a finitely generated submonoid of $(\nn^r,+)$ of rank~$d$, it follows from~\cite[Proposition~2.17]{BG09} that $M'$ is isomorphic to a submonoid of $(\nn^d,+)$. This, in turn, implies that $M$ is isomorphic to a submonoid of $(\nn^d,+)$, which concludes our argument.
\end{proof}

As we are interested in studying monoids satisfying the equivalent conditions of Proposition~\ref{prop:conditions defining monoids in C}, we introduce the following notation.
\medskip

\noindent \textbf{Notation:} Let $\mathcal{C}$ denote the class consisting of all monoids (up to isomorphism) satisfying the conditions in Proposition~\ref{prop:conditions defining monoids in C}. In addition, for every $d \in \nn^\bullet$, we set
\[
	\mathcal{C}_d := \{M \in \mathcal{C} \mid \rk(M) = d \}.
\]

A monoid is \emph{affine} if it is isomorphic to a finitely generated submonoid of the free abelian group $(\zz^d,+)$ for some $d \in \nn$. The interested reader can find a self-contained treatment of affine monoids in~\cite[Part~2]{BG09}. Clearly, the class $\mathcal{C}$ contains a large subclass of affine monoids. Computational aspects of affine monoids and arithmetic invariants of half-factorial affine monoids have been studied in~\cite{GOW19} and~\cite{GOS13}, respectively. Diophantine monoids form a special subclass of that one consisting of affine monoids; Diophantine monoids have been studied in~\cite{CKO02}. Monoids in $\mathcal{C}$ of small rank were recently investigated in~\cite{CO17}. Some other special subclasses of $\mathcal{C}$ have been previously considered in the literature as they naturally arise in the study of algebraic curves, toric geometry, and homological algebra. Here we offer some examples.

\begin{example}
	If $M$ is a finitely primary monoid (see definition in Subsection~\ref{subsec:finitely primary monoids}), then $M$ is primary and satisfies that $\widehat{M} \cong (\nn^d,+)$ \cite[Theorem~2.9.2]{GH06b}. Hence $\mathcal{C}$ contains all finitely primary monoids. The class of finitely primary monoids contains no (nontrivial) finitely generated monoids.
\end{example}

\begin{example} \label{ex:good semigroups}
	Good semigroups, which also form a subclass of $\mathcal{C}$, were first studied in~\cite{fDM87} under the name ``semigroup of values" in the context of singular curves and then generalized in~\cite{BDF00}, where the term ``good semigroup" was coined. Good semigroups are submonoids of $(\nn^d,+)$ that naturally generalize semigroups of values of an algebraic curve in the sense that monoids on both classes satisfy certain common ``good" properties. For instance, the semigroup $S$ of values of the commutative ring
	\[
		R := \cc[[x,y]] / \big( (x^7 - x^6 + 4x^5y + 2x^3y^2 - y^4 )(x^3 - y^2) \big)
	\]
	is represented in Figure~\ref{fig:good monoid}.
	\begin{figure}[h]
		\begin{center}
			\includegraphics[width=8.5cm]{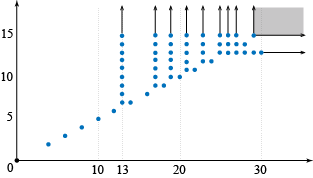}
			\caption{A non-finitely generated good semigroup $S$ (note that the first infinite column contains infinitely many atoms).}
			\label{fig:good monoid}
		\end{center}
	\end{figure}
	As $\{(x,y) \in S \mid x < 13\}$ is finite, the affine line $x=13$ of~$\rr^2$ contains infinitely many atoms of~$S$. Hence the good semigroup~$S$ is not finitely generated (for more details on this example, see~\cite[page~8]{BDF05}). In addition, it has been verified in~\cite[Example~2.16]{BDF00} that the good semigroup with underlying set
	\[
		\{(0,0)\} \cup \{(x,y) \in \nn^2 \mid x \ge 25 \ \text{and} \ y \ge 27 \}
	\]
	is not the semigroup of values of any algebraic curve. Good semigroups have received substantial attention since they were introduced; see for example~\cite{BDF00,BDF05,mD97}, and see~\cite{DGMT18,MZ19} for more recent studies.
\end{example}

A subclass of non-finitely generated monoids in $\mathcal{C}$ naturally shows in connection with the Huneke-Wiegand conjecture.

\begin{example} \label{ex:Leamer monoids}
	From the structure theorem for modules over a PID, one obtains that if $R$ is a $1$-dimensional integrally-closed local domain and~$M$ is a finitely generated torsion-free $R$-module, then~$M$ is free if and only if $M \otimes_R \text{Hom}(M, R)$ is torsion-free. In the same direction, Huneke and Wiegand claimed in~\cite[pages 473--474]{HW94} that if~$R$ is a $1$-dimensional Gorenstein domain and~$M$ is a nonzero finitely generated non-projective $R$-module, then $M \otimes_R \text{Hom}(M,R)$ has a nontrivial torsion submodule (this is known as the Huneke-Wiegand Conjecture). Given a numerical monoid $\Gamma$ and $s \in \nn \setminus \Gamma$, consider the set
	\[
		M^s_\Gamma := \{(0, 0)\} \cup \{(x, n) \mid \{x, x+s, x+2s, \dots, x + ns\} \subseteq \Gamma \} \subset \nn^2,
	\]
	which parameterizes all arithmetic progressions of step size $s$ contained in the monoid~$\Gamma$. Clearly, $M^s_\Gamma$ is a submonoid of $(\nn^2,+)$; it is called the \emph{Leamer monoid} of $\Gamma$ of \emph{step size} $s$. Figure~\ref{fig:Leamer monoid} shows the Leamer monoid of the numerical monoid $\langle 11,12,15 \rangle$ of step size $2$.
	\begin{figure}[h]
		\begin{center}
			\includegraphics[width=8.5cm]{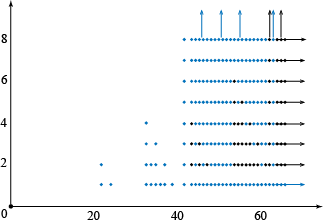}
			\caption{The Leamer monoid $M_{\Gamma}^s$ consisting of all arithmetic progressions in $\Gamma = \langle 11,12,15 \rangle$ of step size $s=2$. The atoms of $M_{\Gamma}^s$ are the light blue dots while the non-atoms are the black dots.}
			\label{fig:Leamer monoid}
		\end{center}
	\end{figure}
	The atomic structure of Leamer monoids is connected to the Huneke-Wiegand Conjecture via~\cite[Corollary~7]{GL13}. Leamer monoids are non-finitely generated \cite[Lemma~2.8]{HKMOP14} rank-$2$ monoids contained in the class~$\mathcal{C}$. Arithmetic properties of Leamer monoids have been studied in~\cite{HKMOP14} and, more recently, in~\cite{CT19}.
\end{example}

The following example has been kindly provided by Roger Wiegand and should appear in the forthcoming manuscript~\cite{AGW19}.

\begin{example} \label{ex:irrational cones}
	Let $\alpha$ and $\beta$ be two positive irrational numbers such that $\alpha < \beta$, and consider the monoid $M_{\alpha, \beta}$ defined as follows:
	\[
		M_{\alpha,\beta} := \{(0,0)\} \bigcup \bigg\{ (m,n) \in \nn^2 \ \bigg{|} \ \alpha < \frac{n}{m} < \beta \bigg\}.
	\]
	Farkas-Minkowski-Weyl Theorem guarantees that $M_{\alpha, \beta}$ is not finitely generated and, therefore, $|\mathcal{A}(M_{\alpha, \beta})| = \infty$. In addition, $M_{\alpha, \beta}$ is a primary FFM (see Proposition~\ref{prop:monoids in C are FF monoids} and Proposition~\ref{prop:primary characterization}). The sequence of monoids $(M_n)_{n \ge 3}$ obtained by setting
	\[
		\alpha = \frac{2}{n + \sqrt{n^2 - 4}} \quad \text{ and } \quad \beta = \frac{n + \sqrt{n^2 - 4}}{2}
	\]
	shows up in the study of Betti tables of short Gorenstein algebras. In an ongoing project, Avramov, Gibbons, and Wiegand have proved that for each $n \in \nn_{\ge 3}$,
	\[
		\mathcal{A}(M_n) = \{\omega_n^{1-a}(1,b) \mid (a,b) \in \Gamma\},
	\]
	where $\omega_n \colon (p,q) \mapsto (np - q,p)$ is an automorphism of $M_n$ and $\Gamma := \zz \times \ldb 1,n-2 \rdb$. This suggests the following question.
	
	\begin{question}
		For any irrational (or irrational and algebraic) numbers $\alpha$ and $\beta$ with $\alpha < \beta$, can we generalize Example~\ref{ex:irrational cones} to describe the set of atoms of $M_{\alpha, \beta}$?
	\end{question}
\end{example}
\medskip

\subsection{The Cones of Monoids in $\mathcal{C}$} Let $\ff$ be an ordered field. We say that a monoid~$M$ is \emph{embedded} into a vector space $V$ over $\ff$ if $M$ is a submonoid of the additive group~$V$. As mentioned in the introduction, a monoid $M$ in $\mathcal{C}$ can be embedded into the finite-dimensional $\ff$-vector space $\ff \otimes_\zz \gp(M)$ via
\begin{equation} \label{eq:main embedding}
	M \hookrightarrow \gp(M) \hookrightarrow \ff \otimes_\zz \gp(M),
\end{equation}
where the flatness of $\ff$ as a $\zz$-module ensures the injectivity of the second homomorphism. Thus, one can always think of a monoid $M$ in $\mathcal{C}$ being embedded into $\ff \otimes_\zz \gp(M)$ and, therefore, one can refer to the conic hull of $M$.

 Let $V$ be a finite-dimensional vector space over $\ff$. A cone $C$ in~$V$ is called \emph{positive} with respect to a basis $\beta$ of $V$ provided that each $v \in C$ can be expressed as a nonnegative linear combination of the vectors in $\beta$. When $V = \ff^d$ for some $d \in \nn^\bullet$, we say that $C$ is \emph{positive} in $V$ (without specifying any basis) if $C$ is a positive cone with respect to the basis consisting of the canonical vectors $e_1, \dots, e_d$.

\begin{lemma} \label{lem:cones of monoids in C are positive}
	Let $V$ be a finite-dimensional vector space over an ordered field $\ff$, and let $M$ be a monoid in $\mathcal{C}$ embedded into $V$. Then there exists a basis $\beta$ of $V$ such that $\cone_V(M)$ is positive with respect to~$\beta$.
\end{lemma}

\begin{proof}
	Set $d := \rank(M)$. Since $M$ is in $\mathcal{C}_d$, Proposition~\ref{prop:conditions defining monoids in C} guarantees the existence of a submonoid $M'$ of $(\nn^d,+)$ and a monoid isomorphism $\varphi \colon M' \to M$. Extending $\varphi$ to the Grothendieck groups of $M'$ and $M$, and then tensoring with the $\zz$-module $\ff$, we obtain a linear transformation $\bar{\varphi} \colon \ff^d \to \aff_V(M)$ that extends $\varphi$. It is clear that $\bar{\varphi}(\cone_{\ff^d}(M')) = \cone_V(M) \subseteq \aff_V(M)$. Because $\cone_{\ff^d}(M')$ is a positive cone in $\ff^d$, it follows that $\cone_V(M)$ is positive in $\aff_V(M)$ with respect to the basis $\bar{\varphi}(e_1), \dots, \bar{\varphi}(e_d)$. Finally, extending the basis $\bar{\varphi}(e_1), \dots, \bar{\varphi}(e_d)$ to a basis $\beta$ of $V$ one obtains that $\cone_V(M)$ is positive in $V$ with respect to $\beta$.
\end{proof}

The cones of monoids in $\mathcal{C}$ are pointed, as we argue now.

\begin{prop} \label{prop:cones of members of C are pointed}
	Let $V$ be a finite-dimensional vector space over an ordered field $\ff$, and let $M$ be a monoid in $\mathcal{C}$ embedded into $V$. The following statements hold.
	\begin{enumerate}
		\item The cone $\cone_V(M)$ is pointed.
		\smallskip
		
		\item If $\ff$ is an intermediate field of the extension $\qq \subseteq \rr$, then the cone $\overline{\cone_V(M)}$ is pointed.
	\end{enumerate} 
\end{prop}

\begin{proof}
	To argue~(1) suppose, without loss of generality, that $\dim V = d$. By Lemma~\ref{lem:cones of monoids in C are positive} there exists a basis $\beta$ of $V$ such that $\cone_V(M) \subseteq \cone_V(\beta)$. Clearly, $\cone_V(\beta)$ is a simplicial cone and, therefore, it does not contain any one-dimensional subspace of $V$. This implies that $\cone_V(M)$ does not contain any one-dimensional subspace of $V$, and so $\cone_V(M)$ is pointed.
	\smallskip
	
	To argue~(2) we first notice that (with notation as in the previous paragraph) $\cone_V(\beta)$ is a closed cone of $V$ containing $\cone_V(M)$. So $\overline{\cone_V(M)} \subseteq \cone_V(\beta)$. As observed in the previous paragraph, the fact that $\cone_V(\beta)$ is pointed forces $\overline{\cone_V(M)}$ to be pointed.
\end{proof}

A \emph{lattice} is a partially ordered set $L$, in which every two elements have a unique \emph{join} (i.e., least upper bound) and a unique \emph{meet} (i.e., greatest lower bound). The lattice~$L$ is \emph{complete} if each $S \subseteq L$ has both a join and a meet. Two complete lattices are \emph{isomorphic} if there exists a bijection between them that preserves joins and meets. For background information on (complete) lattices and lattice homomorphisms, see~\cite[Chapter~2]{DP02}.

Let $V$ be a finite-dimensional vector space over an ordered field $\ff$. For a cone $C$ in~$V$, the set of faces of~$C$, denoted by $\mathsf{F}(C)$, is a complete lattice (under inclusion)~\cite[page~164]{rtR70}, where the meet is given by intersection and the join of a given set of faces is the smallest face in $\mathsf{F}(C)$ containing all of them. The lattice $\mathsf{F}(C)$ is called the \emph{face lattice} of $C$. Two cones $C$ and $C'$ are \emph{combinatorially equivalent} provided that their face lattices are isomorphic.

It turns out that the combinatorial and geometric structures of $\cone_V(M)$ do not depend on the embedding proposed in~(\ref{eq:main embedding}).

\begin{prop} \label{prop:combinatorial and geometric equivalence of cones}
	Let $V$ and $V'$ be finite-dimensional vector spaces over an ordered field $\ff$, and let $M$ and $M'$ be isomorphic monoids in $\mathcal{C}$ embedded into $V$ and $V'$, respectively. The following statements hold.
	\begin{enumerate}
		\item The cone $\cone_V(M)$ is affinely equivalent to the cone $\cone_{V'}(M')$.
		\smallskip
		
		\item The cone $\cone_V(M)$ is combinatorially equivalent to the cone $\cone_{V'}(M')$.
		\smallskip
		
		\item The cone $\cone_V(M)$ is homeomorphic to the cone $\cone_{V'}(M')$ when $\qq \subseteq \ff \subseteq \rr$.
	\end{enumerate}
\end{prop}

\begin{proof}
	For cones, being affinely equivalent, being combinatorially equivalent, and being homeomorphic are all equivalence relations. Hence it suffices to argue the proposition assuming that $M$ is a submonoid of $(\nn^d,+)$ and that $V = \ff^d$ for some $d \in \nn^\bullet$.
	
	Since $M \subseteq \nn^d$, we can further assume that $\gp(M) \subseteq \zz^d$. Let $\varphi \colon M \to M'$ be a monoid isomorphism. Extending $\varphi$ to the Grothendieck groups of $M$ and $M'$, and then tensoring with the $\zz$-module $\ff$, one obtains a linear transformation $\bar{\varphi} \colon V \to \ff \otimes_\zz \gp(M') \subseteq V'$ that extends $\varphi$. Since $\ff$ is a flat $\zz$-module, $\ker \bar{\varphi}$ is trivial and, therefore, $\bar{\varphi}$ is an injective linear transformation satisfying that
	\[
		\bar{\varphi}(\cone_V(M)) = \cone_{\bar{\varphi}(V)}(M') = \cone_{V'}(M').
	\]
	As a consequence, the cones $\cone_V(M)$ and $\cone_{V'}(M')$ are affinely equivalent, from which~(1) follows.
	\smallskip
	
	To argue~(2) it suffices to observe that because the map $\bar{\varphi}$ is a linear bijection taking $\cone_V(M)$ onto $\cone_{V'}(M')$, the assignment $F \mapsto \bar{\varphi}(F)$ for any face $F$ of $\mathsf{F}(\cone_V(M))$ is an order-preserving bijection from the face lattice $\mathsf{F}(\cone_V(M))$ to the face lattice $\mathsf{F}(\cone_{V'}(M'))$.
	\smallskip
	
	To verify~(3), suppose that $\qq \subseteq \ff \subseteq \rr$. Then part~(1) guarantees the existence of an invertible linear transformation $\bar{\varphi} \colon \aff_V(\cone_V(M)) \to \aff_{V'}(\cone_{V'}(M'))$ mapping $\cone_V(M)$ to $\cone_{V'}(M')$. As $\qq \subseteq \ff \subseteq \rr$, the subspace $\ff^d$ is dense in $\rr^d$ with respect to the Euclidean topology. Thus, the map $\bar{\varphi}$ must be a homeomorphism and, therefore, $\cone_{V}(M)$ and $\cone_{V'}(M')$ are homeomorphic.
\end{proof}

For a monoid $M$ in $\mathcal{C}$ and a finite-dimensional vector space $V$ over an ordered field~$\ff$, Proposition~\ref{prop:combinatorial and geometric equivalence of cones} guarantees that the geometric and combinatorial aspects, and the topological aspects (when $\qq \subseteq \ff \subseteq \rr$) of $\cone_V(M)$ do not depend on $V$ but only on $\ff$. As a consequence, we shall sometimes write $\cone_\ff(M)$ instead of $\cone_V(M)$, choosing without specifying the finite-dimensional vector space over $\ff$ where the conic hull of~$M$ is embedded into.

\begin{cor} \label{cor:rank equal dimension}
	Let $\ff$ be an ordered field, and let $M$ be a monoid in $\mathcal{C}$. Then the equality $\dim \cone_\ff(M) = \emph{rank}(M)$ holds.
\end{cor}

\begin{proof}
	Consider the $\ff$-vector space $V := \ff \otimes_{\zz} \gp(M)$. As the dimension of the $\qq$-vector space $\qq \otimes_{\zz} \gp(M)$ equals $\rk(M)$, and the vector space $V$ can be obtained from $\qq \otimes_\zz \gp(M)$ by extending scalars from $\qq$ to $\ff$, the equality $\dim V = \rank(M)$ holds. Therefore
	\[
		\dim \cone_V(M) = \dim \aff_V(M) = \dim \aff_V(\gp(M)) = \dim \ff \otimes_{\zz} \gp(M) = \rk(M).
	\]
\end{proof}

Members of $\mathcal{C}$ are finite-rank torsion-free monoids. However, not every finite-rank torsion-free monoid is in $\mathcal{C}$. The next two examples shed some light upon this observation.

\begin{example}
	A nontrivial submonoid $M$ of $(\qq_{\ge 0},+)$ is obviously a rank-$1$ torsion-free monoid. It follows from Proposition~\ref{prop:conditions defining monoids in C} that $M$ belongs to $\mathcal{C}$ if and only if $M$ is isomorphic to a numerical monoid. Hence~\cite[Proposition~3.2]{fG17} guarantees that~$M$ is in~$\mathcal{C}$ if and only if $M$ is finitely generated. As a result, non-finitely generated submonoids of $(\qq_{\ge 0},+)$ such as $\langle 1/p \mid p \text{ is prime} \rangle$ are finite-rank torsion-free monoids that do not belong to the class $\mathcal{C}$. The atomic structure of submonoids of $(\qq_{\ge 0},+)$ have been studied in \cite{fG19,fG19b} and some of their factorization invariants have been investigated in~\cite{CGG20,GO17}. In addition, their monoid algebras were recently considered in~\cite{CG19,fG20}.
\end{example}

Clearly, the Grothendieck group of a non-finitely generated submonoid of $(\qq_{\ge 0},+)$ cannot be free. The following example, courtesy of Winfried Bruns, shows that a finite-rank torsion-free monoid may not belong to $\mathcal{C}$ even though its Grothendieck group is free.

\begin{example}
	Consider the additive monoid
	\[
		M := \{(0,0)\} \cup \{(m,n) \in \zz^2 \mid n > 0\} \subseteq \zz^2.
	\]
	Note that $M$ is a submonoid of $(\zz^2,+)$ and, therefore, it has rank at most~$2$. In addition, it is clear that $M$ is torsion-free. On the other hand,
	\[
		\overline{\cone_{\rr^2}(M)} = \{(x,y) \in \rr^2 \mid y \ge 0\},
	\]
	which is the upper closed half-space. Thus, the cone $\overline{\cone_{\rr^2}(M)}$ is not pointed. As a consequence, it follows from Proposition~\ref{prop:cones of members of C are pointed} that $M$ does not belong to~$\mathcal{C}$.
\end{example}

\smallskip
%%%%%%%%%%%%%%%%%%%%%%%%%%%%%%%%%%
\subsection{Cones Realized by Monoids in $\mathcal{C}$}
\label{sec:cones realized}

Our next goal is to characterize the positive cones that can be realized as conic hulls of monoids in~$\mathcal{C}$. First, let us argue the following lemma.

\begin{lemma} \label{lem:polyhedral cone containing an interior ray of a convex cone}
	Let $\ff$ be an intermediate field of the extension $\qq \subseteq \rr$, and let $C$ be a $d$-dimensional positive cone in $\ff^d$ for some $d \in \nn$. For each $x \in \inter \, C$, there exists a $d$-dimensional rational simplicial cone $C_x$ such that $\ff_{> 0} x \subseteq \inter \, C_x$ and $C_x \subseteq \{0\} \bigcup \inter \, C$.
\end{lemma}

\begin{proof}
	Take $x \in \inter \, C$. For $d=1$, it suffices to set $C_x = \ff_{\ge 0} x$. Then suppose that $d \ge 2$ and write $x = (x_1, \dots, x_d) \in \ff^d$. As~$C$ is positive and $x \in \inter \, C$, one sees that $x_i > 0$ for every $i \in \ldb 1,d \rdb$. Let $\ell$ be the distance from~$x$ to the complement of $\inter \, C$. Since $\inter \, C$ is open, $\ell > 0$. Consider the $d$-dimensional regular simplex
	\[
		\Delta_d := \ff_{\ge 0}^d \bigcap \bigg\{ (y_1, \dots, y_d) \in \ff^d \ \bigg{|} \ \sum_{i=1}^d y_i  \le 1 \bigg\},
	\]
	and then choose $N \in \nn$ sufficiently large such that $\text{diam}(\Delta_d/N) = \sqrt{2}/N < \ell$. In addition, take $q = (q_1,\dots, q_d) \in \qq^d_{> 0}$ such that $q_i < x_i $ for every $i \in \ldb 1,d \rdb$ and $\sum_{i=1}^d (x_i - q_i) < 1/N$. Now set $\Delta := q + \Delta_d/N$. Clearly, $x-q$ is an interior point of $\Delta_d/N$ and, therefore, $x \in \inter \, \Delta$. This, along with the fact that $\text{diam}(\Delta) = \text{diam}(\Delta_d/N) < \ell$, ensures that $\Delta \subset \inter \, C$. Take $C_x := \cone_{\ff^d} (\Delta)$ and observe that $C_x$ is a polyhedral cone contained in $\{0\} \cup \inter \, C$. In addition, $x \in \inter \, \Delta$ implies that $\ff_{> 0}x \subset \inter \, C_x$. As $\dim \, \Delta = d$, we find that $\dim \, C_x = d$. Hence the set of $1$-dimensional faces of the polyhedral cone $C_x$ has size at least $d$. On the other hand, the $1$-dimensional faces of~$C_x$ are determined by some of the vertices of $\Delta$. As the vertex $q$ of $\Delta$ is contained in $\inter \, C_x$, the $1$-dimensional faces of $C_x$ are precisely the $d$ nonnegative rays containing the points $q + e_i/N$ for each $i \in \ldb 1,d \rdb$. Thus, $C_x$ is a $d$-dimensional rational simplicial cone.
\end{proof}

\noindent \textbf{Notation:} Let $\ff$ be an ordered field (whose prime field must be $\qq$), and take $d \in \nn$. Then we call a $1$-dimensional subspace of $\ff^d$ (resp., an infinite ray) a \emph{rational line} (resp., a \emph{rational ray}) if it contains a nonzero vector of $\qq^d$.
\smallskip

\begin{theorem} \label{thm:a characterization of cones generated by monoids in C}
	Let $\ff$ be an intermediate field of the extension $\qq \subseteq \rr$, and let $C$ be a $d$-dimensional positive cone in $\ff^d$ for some $d \in \nn$. Then $C$ can be generated by a submonoid of $(\nn^d,+)$ if and only if each $1$-dimensional face of $C$ is a rational ray.
\end{theorem}

\begin{proof}
	Set $V := \ff^d$. For the direct implication, suppose that $C = \cone_V(M)$, where~$M$ is a maximal rank submonoid of $(\nn^d,+)$. Let $L$ be a $1$-dimensional face of $C$, and let~$x$ be a nonzero point in $L$. Now take $c_1, \dots, c_k \in \ff_{> 0}$ and $x_1, \dots, x_k \in M^\bullet$ such that $x = \sum_{i=1}^k c_i x_i$. If $k=1$, then $x_1 = \frac{1}{c_1} x \in L$ and, therefore,~$L$ is a rational ray. If $k > 1$, then after setting $x'_1 := \sum_{i=2}^k c_i x_i$ we can see that $x'_1 \in \cone_V(M)^\bullet$ and
	\[
		\frac{c_1}{1 + c_1} x_1 + \frac{1}{1 + c_1} x'_1 = \frac{1}{1 + c_1} x \in L.
	\]
	As $L$ is a face of $C$ and the open line segment from $x_1$ to $x'_1$ intersects $L$, the closed line segment from $x_1$ to $x'_1$ must be contained in $L$. In particular, $x_1 \in L$. Hence $L$ is a rational ray.
	\smallskip
	
	For the reverse implication, assume that all $1$-dimensional faces of $C$ are rational rays. Consider the set $M := C \cap \nn^d$. Clearly, $M$ is a submonoid of $(\nn^d,+)$ and $\cone_V(M) \subseteq C$. Take $x \in C^\bullet$, and set $\ell := \ff_{> 0} x$. As the nonempty intersection of faces of $C$ is again a face of $C$, there exists a minimal (under inclusion) face $C'$ of $C$ containing~$x$. Suppose, for the sake of a contradiction, that there exists a supporting hyperplane~$H$ of $C'$ such that $x \in H$ but $C' \not\subseteq H$. In this case, $C' \cap H$ would be a face of $C'$ containing~$x$. However, notice that $C' \cap H$ is a face of $C$ containing $x$ that is strictly contained in $C'$, contradicting the minimality of $C'$. Hence each supporting hyperplane of $C'$ containing $x$ also contains~$C'$. This implies that $x$ is in the relative interior of $C'$. Suppose now that~$C'$ has dimension $d'$. Then by Lemma~\ref{lem:polyhedral cone containing an interior ray of a convex cone} there exists a rational cone $C_x$ with $C^\bullet_x \subseteq \relin \, C'$ and~$d'$ $1$-dimensional faces such that $\ell \subset \relin \, C_x$. Now take $v_1, \dots, v_{d'} \in \nn^d \setminus \{0\}$ such that the $1$-dimensional faces of $C_x$ are precisely the rays $\ff_{\ge 0} v_1, \dots, \ff_{\ge 0} v_{d'}$. As $x \in \relin \, C_x$ we can write $x = \sum_{i=1}^{d'} c_i v_i$ for some $c_1,\dots, c_{d'} \in \ff_{> 0}$. Because $v_i \in C \cap \nn^d$ for each $i \in \ldb 1,d' \rdb$, it follows that $x \in \cone_V(M)$. Hence $C \subseteq \cone_V(M)$, which completes the proof.
\end{proof}

The following example illustrates a positive cone in $\rr^d$ that cannot be generated by any monoid in $\mathcal{C}$.

\begin{example}
	Let $C$ be the cone in $\rr^d$ generated by the set $\{e_1, \dots, e_{d-1}, v_d \}$, where $v_d := \pi e_d + \sum_{i=1}^{d-1} e_i$. It is clear that $C$ is a positive cone. Note, in addition, that $\rr_{\ge 0} v_d$ is a $1$-dimensional face of $C$. Finally, observe that $\rr_{> 0} v_d$ contains no point with rational coordinates. Hence it follows from Theorem~\ref{thm:a characterization of cones generated by monoids in C} that $C$ cannot be generated by any monoid in $\mathcal{C}$.
\end{example}

\medskip
%%%%%%%%%%%%%%%%%%%%%%%%%%%%%
%%%%%%%%%%%%%%%%%%%%%%%%%%%%%
\section{Faces and  Divisor-Closed Submonoids}
\label{sec:divisor-closed submonoids and finitely primary monoids}

\smallskip
%%%%%%%%%%%%%%%%%%
\subsection{Face Submonoids}

Let $\ff$ be an ordered field, and let $M$ be a monoid in $\mathcal{C}$. We would like to understand the structure of the face lattice of $\cone_\ff(M)$ in connection with the divisibility aspects of $M$. In particular, the submonoids of $M$ obtained by intersecting $M$ with the faces of $\cone_\ff(M)$ are relevant in this direction as they inherit many divisibility and atomic properties from $M$.

\begin{definition}
	Let $\ff$ be an ordered field, and let $M$ be a nontrivial monoid in $\mathcal{C}$. A submonoid $N$ of $M$ is called a \emph{face submonoid} of $M$ with respect to $\ff$ provided that $N = M \cap F$ for some face $F$ of $\cone_\ff(M) \subseteq \ff \otimes_\zz \gp(M)$.
\end{definition}

Being a face submonoid of $M$ with respect to $\ff$ does not depend on the way $M$ is embedded into $\ff \otimes_\zz \gp(M)$, as the following proposition indicates.

\begin{prop} \label{prop:face submonoids do not depend on the embedding}
	Let $V$ be a finite-dimensional vector space over an ordered field $\ff$, and let $M$ be a monoid in $\mathcal{C}$ embedded into $V$. The following statements hold.
	\begin{enumerate}
		\item A submonoid $N$ of $M$ is a face submonoid with respect to $\ff$ if and only if there exists a face $F$ of $\cone_V(M)$ such that $N = M \cap F$.
		\smallskip
		
		\item Face submonoids with respect to $\ff$ are preserved by monoid isomorphisms. In particular, face submonoids do not depend on the embedding $M \hookrightarrow V$.
		\smallskip
		
		\item If $N$ is a face submonoid of $M$ with respect to $\ff$, and $F$ is a face of $\cone_V(M)$ satisfying that $N = M \cap F$, then $\cone_V(N) = F$. In particular, there exists a unique face of $\cone_V(M)$ determining the face submonoid $N$.
	\end{enumerate}
\end{prop}

\begin{proof}
	Set $V' := \ff \otimes_\zz \gp(M)$. By the universality property of the Grothendieck group, the inclusion $M \hookrightarrow V$ extends to an injective group homomorphism $\varphi \colon \gp(M) \to V$ and, after tensoring with the flat $\zz$-module $\ff$, one obtains an injective linear transformation $\bar{\varphi} \colon V' \to V$. For the direct implication of~(1), suppose that $N$ is a face submonoid of~$M$, and let $F'$ be a face of $\cone_{V'}(M)$ such that $N = M \cap F'$. Clearly, $F := \bar{\varphi}(F')$ is a face of $\cone_V(M)$. In addition,
	\[
		M \cap F = M \cap \bar{\varphi}(F') = \bar{\varphi}(M \cap F') = \bar{\varphi}(N) = N.
	\]
	For the reverse implication, we observe that if $F$ is a face of $\cone_V(M)$ such that $N = M \cap F$, then $N = \bar{\varphi}^{-1}(N) = M \cap \bar{\varphi}^{-1}(F)$. As $\bar{\varphi}^{-1}(F)$ is a face of $\cone_{V'}(M)$, one finds that $N$ is a face submonoid of~$M$.
	\smallskip
	
	In order to argue~(2), suppose that $M'$ is a monoid in $\mathcal{C}$, and let $\phi \colon M \to M'$ be a monoid isomorphism. As we did in the previous paragraph, we can extend the isomorphism $\phi$ to a bijective linear transformation $\bar{\phi} \colon W \to V'$, where $W = \aff_V(M)$ and $V' = \ff \otimes_\zz \gp(M)$. Now take $N$ to be a face submonoid of $M$ with respect to $\ff$. By the previous part, there exists a face $F$ of $\cone_W(M)$ such that $N = M \cap F$. Since $\bar{\phi}(N) = \bar{\phi}(M) \cap \bar{\phi}(F) = M' \cap \bar{\phi}(F)$ and $\bar{\phi}(F)$ is a face of $\cone_{V'}(M)$, it follows that $\bar{\phi}(N)$ is a face submonoid of $M'$.
	\smallskip
	
	Finally, we prove~(3). Suppose that $F$ is a face of $\cone_V(M)$ such that $N = M \cap F$. Since $N \subseteq F$ and $F$ is a cone, $\cone_V(N) \subseteq F$. To show the reverse inclusion, take $x \in F^\bullet$. Then $x = \sum_{i=1}^k f_i x_i$ for $x_1, \dots, x_k \in M^\bullet$ and $f_1, \dots, f_k \in \ff_{> 0}$. If $k = 1$, then $x_1 = x/f_1 \in F^\bullet$. Otherwise, after setting $x'_1 := \sum_{i=2}^k f_i x_i \in \cone_V(M)$ and $\alpha := \frac{1}{1 + f_1}$, we see that $(1 - \alpha) x_1 + \alpha x_1' = \alpha x \in F^\bullet$. Because $F$ is a face of $\cone_V(M)$ intersecting the open line segment from $x_1$ to $x_1'$, both $x_1$ and $x'_1$ must belong to $F^\bullet$. Now a simple inductive argument shows that $x_1, \dots, x_k \in F^\bullet$. Therefore $x_i \in M \cap F = N$ for each $i \in \ldb 1,k \rdb$. This implies that $x \in \cone_V(N)$.
\end{proof}
\medskip

The set of atoms of face submonoids can be nicely described in terms of the corresponding faces.

\begin{prop} \label{prop:atoms of face submonoids}
	Let $V$ be a finite-dimensional vector space over an ordered field~$\ff$, and let $M$ be a monoid in $\mathcal{C}$ embedded into $V$. If $N$ is a face submonoid of $M$ determined by a face $F$ of $\cone_V(M)$, then $\mathcal{A}(N) = \mathcal{A}(M) \cap F$.
\end{prop}

\begin{proof}
	Since $\mathcal{A}(M) \cap F \subseteq N$, one obtains that $\mathcal{A}(M) \cap F \subseteq \mathcal{A}(N)$. To verify the reverse inclusion, take $a \in \mathcal{A}(N)$ and note that $a/2 \in F$. Now take $b \in \mathcal{A}(M)$ such that $b \mid_M a$, and then take $b' \in M$ such that $a = b+b'$. Because $F$ is a face of $\cone_V(M)$ and $a/2$ belongs to the intersection of $F$ and the open line segment $\{tb + (1-t)b' \mid t \in \ff \text{ and } 0 < t < 1 \}$, the atom $b$ must be contained in~$F$. As a consequence, $b \in \mathcal{A}(M) \cap F \subseteq N$. Therefore $b \mid_M a$ implies that $a = b \in \mathcal{A}(M) \cap F$, which yields the desired inclusion.
\end{proof}

For a monoid $M$ in $\mathcal{C}$, there may be submonoids of $M$ obtained by intersecting $M$ with certain non-supporting hyperplanes whose sets of atoms can still be obtained as in Proposition~\ref{prop:atoms of face submonoids}.

\begin{example}
	Consider the submonoid $M = \langle 2e_1, 2e_2, e_1 + e_2 \rangle$ of $(\nn^2,+)$. It can be readily checked that $\mathcal{A}(M) = \{2e_1, 2e_2, e_1 + e_2\}$. Now set $N = M \cap H$, where $H$ is the hyperplane $\rr(e_1 + e_2)$ of~$\rr^2$. It is clear that $N$ is a submonoid of $M$ satisfying that $\mathcal{A}(N) = \{e_1 + e_2\} = \mathcal{A}(M) \cap H$. However, $N$ is not a face submonoid of $M$.
\end{example}

\smallskip
%%%%%%%%%%%%%%%%%%%%%%%%%%%%%%
\subsection{Characterization of Face Submonoids}

A submonoid $N$ of a monoid $M$ is said to be \emph{divisor-closed} provided that for all $y \in N$ and $x \in M$ the condition $x \mid_M y$ implies that $x \in N$. For any monoid $M$ in $\mathcal{C}$, the definitions of a face submonoid and a divisor-closed submonoid are equivalent.

\begin{theorem} \label{thm:characterization of divisor-closed submonoids in $C$}
	Let $\ff$ be an ordered field, and let $M$ be a monoid in $\mathcal{C}$. Then a submonoid $N$ of $M$ is divisor-closed in~$M$ if and only if $N$ is a face submonoid of $M$ with respect to $\ff$.
\end{theorem}

\begin{proof}
	Let $k$ be the rank of $M$. Since face submonoids of $M$ do not depend on the finite-dimensional vector space over $\ff$ that $M$ is embedded into, we can assume that $M \subseteq \nn^k \subset V := \ff^k$. We first verify that face submonoids of $M$ with respect to $\ff$ are divisor-closed. To do so, take a face $F$ of $\cone_V(M)$ and set $N := M \cap F$. To argue that $N$ is a divisor-closed submonoid of $M$, take $x \in N$ and $y \in M \setminus \{x\}$ with $y \mid_M x$. Then $x = y + y'$ for some $y' \in M$, which implies that $x/2$ is contained in both $F$ and the open line segment $\{ty +(1-t)y' \mid t \in \ff \text{ and } 0 < t < 1\}$ from $y$ to $y'$. Since $F$ is a face of $\cone_V(M)$, both $y$ and $y'$ belong to $F$ and, therefore, $y \in M \cap F = N$. Hence~$N$ is divisor-closed.
	\smallskip
	
	Let us argue the reverse implication by induction. Notice that when $M$ has rank~$1$, it is isomorphic to a numerical monoid, and the only submonoids of $M$ that are divisor-closed are the trivial and $M$ itself. These are precisely the face submonoids of $M$ in~$V$ corresponding to the origin and $\cone_V(M)$, respectively. Fix now $k \in \nn_{\ge 2}$, and assume that the divisor-closed submonoids of any monoid in $\mathcal{C}$ with rank less than $k$ are face submonoids with respect to $\ff$. Let $M$ be a maximal-rank submonoid of $(\nn^k,+)$, and let $N$ be a submonoid of $M$ that is not a face submonoid.
	
	\smallskip
	CASE 1: $\rk(N) = k$. As $N$ is not a face submonoid of $M$, it follows that $N \neq M$. Take $v \in M \setminus N$ and a basis $v_1, \dots, v_k \in N$ of $V$ with $v = \sum_{i=1}^k q_i v_i$ for scalars $q_1, \dots, q_k$, which are rational numbers by Cramer's rule. We assume that $q_1, \dots, q_j \le 0$ and $q_{j+1}, \dots, q_k > 0$ (not all zeros) for some index $j \in \nn^\bullet$. Then
	\[
		dv + \sum_{i=1}^j (dq_i) v_i = \sum_{i = j+1}^k (dq_i) v_i \in N,
	\]
	where $d$ is the least common multiple of the denominators of all the nonzero $q_i$'s. Since $v \notin N$, the monoid $N$ cannot be divisor-closed.
	
	\smallskip
	CASE 2: $\rk(N) < k$. Take $u \in \qq^k$ such that the hyperplane
	\[
		H := \{h \in V \mid \langle h,u \rangle = 0\}
	\]
	of $V$ contains linearly independent vectors $v_1, \dots, v_{k-1} \in M$ such that $v_1, \dots, v_r \in N$, where $r = \rk(N)$. It is clear that $N \subseteq H$. Now we consider the following two subcases.
	
	\smallskip
	CASE 2.1: $H$ is a supporting hyperplane of the cone $\cone_V(M)$. Consider the face $F := H \cap \cone_V(M)$ of $\cone_V(M)$. It follows from part~(3) of Proposition~\ref{prop:face submonoids do not depend on the embedding} that $\cone_V(M \cap F) = F$. Then each face of $\cone_V(M \cap F)$ is a face of $F$ and, therefore, a face of $\cone_V(M)$. Thus, the fact that $N$ is not a face submonoid of $M$ implies that $N$ is not a face submonoid of $M \cap F$. Because $\rk(M \cap F) < \rk(M)$, our inductive assumption guarantees that~$N$ is not a divisor-closed submonoid of $M \cap F$. Hence $N$ cannot be a divisor-closed submonoid of $M$.
	
	\smallskip
	CASE 2.2: $H$ is not a supporting hyperplane of the cone $\cone_V(M)$. In this case, there exist $w_{r+1}, w'_{r+1} \in M$ such that $\langle w_{r+1}, u \rangle > 0$ and $\langle w'_{r+1}, u \rangle < 0$. Because $\{v_1, \dots, v_{k-1}, w'_{r+1}\}$ is a basis of $V$ and $N$ is a divisor-closed submonoid of $M$, there exists $w_{r+2} \in M$ contained in $\cone_V(v_1, \dots, v_{k-1}, w'_{r+1}) \setminus H$ and satisfying that $S := \{v_1, \dots, v_r\} \cup \{w_{r+1}, w_{r+2}\}$ is linearly dependent. Clearly, $\langle w_{r+2}, u \rangle  < 0$. After relabeling the vectors $v_1, \dots, v_r$ (if necessary) and using Cramer's rule, we find that
	\begin{equation} \label{eq:linearly dependence}
		\sum_{i=1}^j q_i v_i = \bigg( \sum_{i=j+1}^r q_i v_i \bigg) + q_{r+1} w_{r+1} + q_{r+2} w_{r+2} 
	\end{equation}
	for some index $j \in \ldb 1,r \rdb$ and coefficients $q_1, \dots, q_{r+1} \in \qq_{\ge 0}$ and $q_{r+2} \in \qq$ (not all zeros). Observe that $q_{r+1}$ and $q_{r+2}$ are both different from zero. After taking the scalar product with $u$ in both sides of~(\ref{eq:linearly dependence}), one can see that $q_{r+2} \langle w_{r+2}, u \rangle = -q_{r+1} \langle w_{r+1}, u \rangle$. Hence $q_{r+1}$ and $q_{r+2}$ are both positive. Now we can multiply~(\ref{eq:linearly dependence}) by the common denominator $d$ of all the nonzero $q_i$'s, to obtain that $w_{r+1} \mid_M \sum_{i=1}^j (dq_i) v_i$. Since $w_{r+1} \notin N$, the submonoid~$N$ is not divisor-closed.
\end{proof}

We have seen in Proposition~\ref{prop:combinatorial and geometric equivalence of cones} that for an ordered field $\ff$ if a monoid $M$ in $\mathcal{C}$ is embedded into two $\ff$-vector spaces $V$ and $V'$, then $\cone_V(M)$ is combinatorially equivalent to $\cone_{V'}(M)$. Our next goal is to argue that the combinatorial structure of $\cone_V(M)$ does not depend on the chosen ordered field. First, let us argue the following lemma.

\begin{lemma} \label{lem:correspondence between faces and divisor-closed submonoids}
	Let $V$ be a finite-dimensional vector space over an ordered field $\ff$, and let $M$ be a monoid in $\mathcal{C}$ maximally embedded into $V$. If $F$ and $F'$ are two faces of $\cone_V(M)$, then $F \subseteq F'$ if and only if $F \cap M \subseteq F' \cap M$; in particular, $F = F'$ if and only if $F \cap M = F' \cap M$.
\end{lemma}

\begin{proof}
	It suffices to argue the first statement, as the second statement is an immediate consequence of the first one. The direct implication is straightforward. For the reverse implication, assume that $F \cap M \subseteq F' \cap M$. Take $v \in F$. Since $F \subseteq \cone_V(M)$, there exist $f_1, \dots, f_n \in \ff_{> 0}$ and $x_1, \dots, x_n \in M^\bullet$ such that $v = \sum_{i=1}^n f_i x_i$. If $n=1$, then $x_1 = v/f_1 \in F \cap M \subseteq F' \cap M \subseteq F'$ and so $x \in F'$. Otherwise, for every $j \in \ldb 1,n \rdb$ the open segment from $x_j$ to $\sum_{i \neq j} f_i x_i$ intersects $F$ at $\frac{1}{1 + f_j} v$ and, therefore, $x_j \in F$. As $x_1, \dots, x_n \in F \cap M \subseteq F' \cap M \subseteq F'$, it follows that $v \in F'$. Hence $F \subseteq F'$, which completes our argument.
\end{proof}

\begin{prop} \label{prop:the combinatorics of cone(M) does not depend on the field}
	Let $V$ and $V'$ be finite-dimensional vector spaces over ordered fields~$\ff$ and $\mathbb{K}$, respectively. If $M$ and $M'$ are isomorphic monoids in $\mathcal{C}$ embedded into~$V$ and~$V'$, respectively, then $\cone_V(M)$ and $\cone_{V'}(M')$ are combinatorially equivalent.
\end{prop}

\begin{proof}
	Set $d := \rank(M)$. By Proposition~\ref{prop:conditions defining monoids in C} there exists a submonoid $M''$ of $(\nn^d,+)$ that is isomorphic to $M$, i.e., $M'' \subseteq \nn^d \subseteq \ff^d$. Now set $C_\qq := \cone_{\qq^d}(M'')$ and $C_\ff := \cone_{\ff^d}(M'')$.
	
	We first verify that $C_\qq$ and $C_\ff$ are combinatorially equivalent. Clearly, $F \cap \qq^d$ is a face of $C_\qq$ when $F$ is a face of $C_\ff$. Then define $\Phi \colon \mathsf{F}(C_\ff) \to \mathsf{F}(C_\qq)$ via $F \mapsto F \cap \qq^d$, where $\mathsf{F}(C_\ff)$ and $\mathsf{F}(C_\qq)$ are the face lattices of the cones $C_\ff$ and $C_\qq$, respectively. If~$F_1$ and $F_2$ are two faces of $C_\ff$ such that $\Phi(F_1) = \Phi(F_2)$, then
	\[
		F_2 \cap M'' = \Phi(F_2) \cap M'' = \Phi(F_1) \cap M'' = F_1 \cap M''.
	\]
	So it follows from Lemma~\ref{lem:correspondence between faces and divisor-closed submonoids} that $F_1 = F_2$. Thus, $\Phi$ is injective. Similarly one can verify that for any two faces $F_1$ and $F_2$ of $C_\ff$, the inclusion $\Phi(F_1) \subseteq \Phi(F_2)$ holds if and only if $F_1 \subseteq F_2$. As a result, $\Phi$ is a homomorphism of posets. To argue that~$\Phi$ is surjective, let $F'$ be a face of $C_\qq$. As $F' \cap M''$ is a divisor-closed submonoid of~$M''$, there exists a face $F$ of $C_\ff$ such that $F \cap M'' = F' \cap M''$. So we obtain that $\Phi(F) \cap M'' = F \cap M'' = F' \cap M''$, and if follows from Lemma~\ref{lem:correspondence between faces and divisor-closed submonoids} that $\Phi(F) = F'$. Because $\Phi$ is a bijective homomorphism of posets, it is indeed a lattice isomorphism.
	
	Now it follows from part~(2) of Proposition~\ref{prop:combinatorial and geometric equivalence of cones} that $\cone_{V}(M)$ is combinatorially equivalent to $C_\ff$ and, therefore, $\cone_V(M)$ is combinatorially equivalent to $C_\qq$. In a similar manner, one can argue that $\cone_{V'}(M')$ is combinatorially equivalent to $C_\qq$, whence the proposition follows.
\end{proof}

\begin{remark}
	By virtue of Proposition~\ref{prop:the combinatorics of cone(M) does not depend on the field}, in order to study the combinatorial aspects of the cones of a monoid $M$ in $\mathcal{C}$ one can simply embed $M$ into any finite-dimensional vector space over \emph{any} ordered field. In particular, there is no loss in choosing the ordered field to be $\rr$, and from now on we will do so.
\end{remark}

\medskip
%%%%%%%%%%%%%%%%%%%%%
%%%%%%%%%%%%%%%%%%%%%
\section{Geometry and Factoriality}
\label{sec:factoriality}

In this section we study the factoriality of members of $\mathcal{C}$ in connection with the geometric properties of their corresponding conic hulls. We shall provide geometric characterizations of the UFMs, HFMs, and OHFMs in $\mathcal{C}$.

\medskip
%%%%%%%%%%%%%%%%%%%%%%%%%
\subsection{Unique Factorization Monoids}

To begin with, let us characterize the UFMs in $\mathcal{C}$. Recall that a monoid $M$ is a UFM if and only if $|\mathsf{Z}(x)| = 1$ for all $x \in M$.

\begin{theorem} \label{thm:factoriality characterizations}
	Let $M$ be a monoid in $\mathcal{C}$, and set $V = \rr \otimes_{\zz} \emph{\gp}(M)$. The following statements are equivalent.
	\begin{enumerate}
		\item[(a)] The monoid $M$ is a UFM.
		\smallskip
		
		\item[(b)] Each face submonoid of $M$ is a UFM.
		\smallskip
		
		\item[(c)] The equality $|\mathcal{A}(M)| = \dim \, \cone_V(M)$ holds.
	\end{enumerate}
\end{theorem}

\begin{proof}
	Set $d := \dim \cone_V(M)$. By part~(2) of Proposition~\ref{prop:face submonoids do not depend on the embedding}, there is not loss of generality in assuming that $M$ is a submonoid of $(\nn^d,+)$, in which case $V = \rr^d$. It follows from Corollary~\ref{cor:rank equal dimension} that $d = \dim \cone_V(M) = \rank(M)$.
	\smallskip
	
	(b) $\Rightarrow$ (a): It is obvious.
	\smallskip
	
	(a) $\Rightarrow$ (c): We first verify that $|\mathcal{A}(M)| \ge d$. Since $\cone_V(M) = \cone_V(\mathcal{A}(M))$, the atoms in $\mathcal{A}(M)$, when considered as vectors form a generating set of $V$. As a consequence, $|\mathcal{A}(M)| \ge \dim V = d$. To argue the reverse inequality suppose, by way of contradiction, that $|\mathcal{A}(M)| > d$. As the vectors in $\mathcal{A}(M)$ generate $V$, one can take distinct atoms $a_1,\dots, a_{d+1} \in \mathcal{A}(M)$ such that $a_1, \dots, a_d$ are linearly independent. Then there exist coefficients $\beta_1, \dots, \beta_{d+1} \in \rr$ not all zeros such that $\sum_{i=1}^{d+1} \beta_i a_i = 0$. Since $a_1, \dots, a_d$ are linearly independent and $a_1, \dots, a_{d+1} \in \qq^d$, Cramer's rule guarantees that $\beta_1, \dots, \beta_{d+1} \in \qq$. There is no loss in assuming that there is an index $k \in \ldb 1,d \rdb$ such that $\beta_i < 0$ for each $i \in \ldb 1,k \rdb$ and $\beta_i \ge 0$ for each $i \in \ldb k+1,d+1 \rdb$. Hence $\sum_{i=1}^k \beta_i a_i$ and $\sum_{i = k+1}^{d+1} (-\beta_i) a_i$ are two distinct factorizations of the same element of $M$, contradicting that $M$ is a UFM.
	\medskip
	
	(c) $\Rightarrow$ (b): For this implication, suppose that $|\mathcal{A}(M)| = d$. Let $N$ be a face submonoid of $M$. Since $\cone_V(M) = \cone_V(\mathcal{A}(M))$ and $\mathcal{A}(M)$ is finite, Farkas-Minkowski-Weyl Theorem ensures that $\cone_V(M)$ is polyhedral. Then $N = M \cap H$ for some supporting hyperplane $H = \{x \in \rr^d \mid \langle x,u \rangle = 0 \}$ of $\cone_V(M)$ determined by $u \in \rr^d$. Suppose that $\cone_V(M) \subseteq H^-$. Now if $a \in \mathcal{A}(N)$ and $a = x_1 + x_2$ for some $x_1, x_2 \in M$, then $\langle x_1, u \rangle + \langle x_2, u \rangle = 0$ and, therefore, $\langle x_1, u \rangle = \langle x_2, u \rangle = 0$. So $x_1, x_2 \in M \cap H = N$. Because $a \in \mathcal{A}(N)$ either $x_1 = 0$ or $x_2 = 0$, whence $a \in \mathcal{A}(M)$. This implies that $\mathcal{A}(N) = \mathcal{A}(M) \cap N$. Thus, $N$ is a UFM.
\end{proof}

\begin{cor}
	Let $M$ be a monoid in $\mathcal{C}$, and set $V = \rr \otimes_\zz \emph{\gp}(M)$. If $M$ is a UFM, then $\cone_V(M)$ is rational and polyhedral.
\end{cor}

\begin{proof}
	As we did in the proof of Theorem~\ref{thm:factoriality characterizations}, we assume that $M \subseteq \nn^d$ and that $V = \rr^d$. By Theorem~\ref{thm:factoriality characterizations}, the monoid $M$ is finitely generated and so $\cone_V(M)$ is the conic hull of a finite set. Therefore it follows from Farkas-Minkowski-Weyl Theorem that $\cone_V(M)$ is polyhedral. In addition, $\cone_V(M)$ is rational because each of its $1$-dimensional faces contains an element of $\mathcal{A}(M)$.
\end{proof}

\smallskip
%%%%%%%%%%%%%%%%%%%%%
\subsection{Half-Factorial Monoids}

The notion of half-factoriality is a weaker version of the notion of factoriality (or being a UFD). The purpose of this subsection is to offer characterizations of half-factorial monoids in the class $\mathcal{C}$ in terms of their face submonoids and in terms of the convex hulls of their sets of atoms.

\begin{definition}
	An atomic monoid $M$ is called an \emph{HFM} (or a \emph{half-factorial monoid}) if for all $x \in M^\bullet$ and $z, z' \in \mathsf{Z}(x)$, the equality $|z| = |z'|$ holds.
\end{definition}

Half-factoriality was first investigated by L.~Carlitz in the context of algebraic number fields; he proved that an algebraic number field is an HFD (i.e., a half-factorial domain) if and only if its class group has size at most two~\cite{lC60}. However, the term ``half-factorial domain" is due to A.~Zaks~\cite{aZ76}. In~\cite{aZ80}, Zaks studied Krull domains that are HFDs in terms of their divisor class groups. Parallel to this, L.~Skula~\cite{lS76} and J.~\'Sliwa~\cite{jS76}, motivated by some questions of W.~Narkiewicz on algebraic number theory~\cite[Chapter~9]{wN04}, carried out systematic studies of HFDs. Since then HFMs and HFDs have been actively studied (see \cite{CC00} and references therein).
\smallskip

HFMs in $\mathcal{C}$ can be characterized as follows.

\begin{theorem} \label{thm:HF characterizations}
	Let $M$ be a nontrivial monoid in $\mathcal{C}$, and set $V = \rr \otimes_{\zz} \emph{\gp}(M)$. The following statements are equivalent.
	\begin{enumerate}
		\item[(a)] The monoid $M$ is an HFM.
		\smallskip
		
		\item[(b)] Each face submonoid of $M$ is an HFM.
		\smallskip
		
		\item[(c)] The inequality $\dim \, \conv_V(\mathcal{A}(M)) < \dim \, \cone_V(M)$ holds.
	\end{enumerate}
\end{theorem}

\begin{proof}
	By Proposition~\ref{prop:conditions defining monoids in C}, there is not loss in assuming that $M$ is a submonoid of $(\nn^d,+)$, where $d := \rank(M)$. In this case, $V = \rr^d$.
	\smallskip
	
	(a) $\Rightarrow$ (c): Suppose that $M$ is an HFM. Since $\cone_V(\mathcal{A}(M)) = \cone_V(M)$, one can take linearly independent vectors $a_1, \dots, a_d$ in $\mathcal{A}(M)$. Take also $u \in \rr^d$ and $\alpha \in \rr$ such that the polytope $\conv(a_1, \dots, a_d)$ is contained in the affine hyperplane $H := \{r \in \rr^d \mid \langle r,u \rangle = \alpha\}$. In addition, fix $a \in \mathcal{A}(M)$, and write $a = \sum_{i=1}^d q_i a_i$ for some coefficients $q_1, \dots, q_d \in \qq$ (by Cramer's rule, the coefficients can be taken to be rationals). The equality $a = \sum_{i=1}^d q_i a_i$, along with the fact that $M$ is an HFM, guarantees that $\sum_{i=1}^d q_i = 1$. As a result,
	\[
		\langle a, u \rangle = \sum_{i=1}^d q_i \langle a_i, u \rangle = \alpha \sum_{i=1}^d q_i = \alpha,
	\]
	which means that $a \in H$. Hence $\mathcal{A}(M) \subset H$, which implies that $\dim \conv_V(\mathcal{A}(M))$ is at most $d-1$. Thus, $\dim \conv_V(\mathcal{A}(M)) < \dim \, \cone_V(M)$. 
	\smallskip
	
	(c) $\Rightarrow$ (b): Suppose that $\dim \, \conv_V(\mathcal{A}(M)) < \dim \, \cone_V(M)$. Then there exists an affine hyperplane $H$ containing $\conv_V(\mathcal{A}(M))$. As in the previous paragraph, take $u \in \rr^d$ and $\alpha \in \rr$ such that $H = \{r \in \rr^d \mid \langle r, u \rangle = \alpha\}$. Now if $x \in M^\bullet$ and $z = a_1 + \dots + a_\ell$ is a factorization in $\mathsf{Z}(x)$, then
	\[
		\ell = \frac{1}{\alpha} \sum_{i=1}^\ell \langle a_i, u \rangle = \frac{1}{\alpha} \bigg\langle \sum_{i=1}^\ell a_i, u \bigg\rangle = \frac{1}{\alpha} \langle x,u \rangle.
	\]
	Hence $\mathsf{L}(x) = \{\frac{1}{\alpha} \langle x, u \rangle\}$ for all $x \in M^\bullet$, and so $M$ is an HFM.
	\smallskip
	
	(b) $\Rightarrow$ (a): It is straightforward.
\end{proof}

\begin{cor} \label{cor:half-factorial geometric characterization}
	A nontrivial monoid $M$ in $\mathcal{C}$ is an HFM if and only if $\mathcal{A}(M)$ is contained in an affine hyperplane $H$ of $\rr \otimes_\zz \emph{gp}(M)$ not containing the origin, in which case $\mathcal{A}(M) = M \cap H$.
\end{cor}

\begin{proof}
	The equivalence is an immediate consequence of Theorem~\ref{thm:HF characterizations}. It is clear that $\mathcal{A}(M) \subseteq M \cap H$. For the reverse inclusion assume, as we did in the proof of Theorem~\ref{thm:HF characterizations}, that $M \subseteq \nn^d$ and $V = \rr^d$ (where $d$ is the rank of~$M$) and then take $u \in \rr^d$ and $\alpha \in \rr$ such that $H = \{ r \in \rr^d \mid \langle r,u \rangle = \alpha \}$. As $H$ does not contain the origin, $\alpha \neq 0$. If $n \in \nn_{\ge 2}$ and $a_1, \dots, a_n \in \mathcal{A}(M)$, then $\langle \sum_{i=1}^n a_i, u \rangle = \sum_{i=1}^n \langle a_i,u \rangle = n \alpha \neq \alpha$. Hence $M \cap H \subseteq \mathcal{A}(M)$. 
\end{proof}

\begin{remark}
	Fairly similar versions of Corollary~\ref{cor:half-factorial geometric characterization} have been previously established by A.~Zaks in~\cite{aZ80} and by F.~Kainrath and G.~Lettl in~\cite{KL00}.
\end{remark}

For $d \in \nn^\bullet$, the convex hull in $\rr^d$ of finitely many points with integer coordinates is called a \emph{lattice polytope}. 

\begin{example}
	Take $d \in \nn^\bullet$, and let $S$ be a finite subset of $\zz^d$. Let $P$ be the lattice polytope one obtains after taking the convex hull of $S$. Now let $L$ be the set of all points with integer coordinates contained in~$P$, and consider the submonoid $M$ of $(\zz^{d+1},+)$ defined as $M = \langle (p,1) \in \nn^{d+1} \mid p \in L \rangle$. Monoids constructed in this way are called \emph{polytopal affine monoids}. Polytopal affine monoids may not be positive in general. However, they are positive when $S$ is contained in $\nn^d$. If this is the case, then it follows from Theorem~\ref{thm:HF characterizations} that $M$ is an HFM in $\mathcal{C}$, and it follows from Corollary~\ref{cor:half-factorial geometric characterization} that $\mathcal{A}(M) = L$. Figure~\ref{fig:polytopal monoid} illustrates a lattice pentagon in the first quadrant of $\rr^2$ along with the generators in $\nn^3$ of its polytopal affine monoid.
	\begin{figure}[h]
		\begin{center}
			\includegraphics[width=5cm]{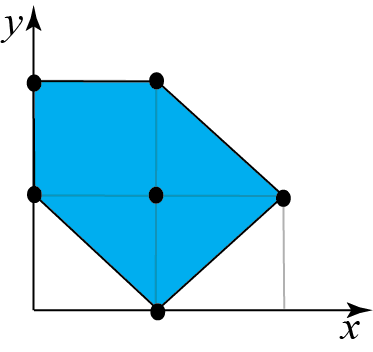}
			\hspace{2cm}
			\includegraphics[width=5.4cm]{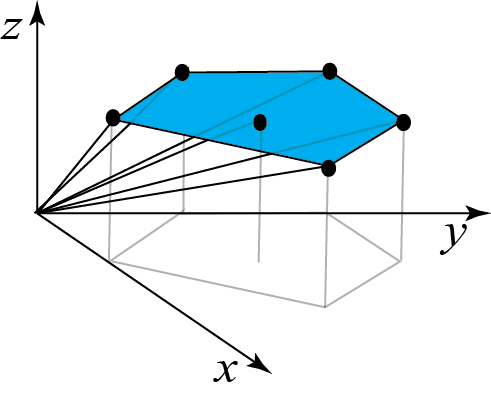}
			\caption{On the left, a lattice pentagon in $\rr^2$. On the right, its polytopal monoid in $\rr^3$.}
			\label{fig:polytopal monoid}
		\end{center}
	\end{figure}
\end{example}

The chain of implications~(\ref{eq:monoid atomicity taxonomy}), where being a UFM, an HFM, and an atomic monoid are included, has received a great deal of attention since it was first studied (in the context of integral domains) by Anderson, Anderson, and Zafrullah~\cite{AAZ90}:
\begin{equation} \label{eq:monoid atomicity taxonomy}
	\text{UFM} \ \Rightarrow \ [\text{HFM}, \text{FFM}] \ \Rightarrow \ \text{BFM} \ \Rightarrow \ \text{ACCP monoid} \ \Rightarrow \ \text{atomic monoid}.
\end{equation}
The first two implications above are obvious, while the last two implications follow from~\cite[Proposition~1.1.4]{GH06b} and~\cite[Corollary~1.3.3]{GH06b}. In addition, none of the implications is reversible, and examples witnessing this observation (in the context of integral domains) can be found in~\cite{AAZ90}. We have already seen that not every monoid in $\mathcal{C}$ is an HFM. However, each monoid in $\mathcal{C}$ is an FFM, as the following proposition illustrates.

\begin{prop} \label{prop:monoids in C are FF monoids}
	Each monoid in $\mathcal{C}$ is an FFM.
\end{prop}

\begin{proof}
	By Proposition~\ref{prop:conditions defining monoids in C}, it suffices to show that for every $d \in \nn^\bullet$, any submonoid~$M$ of $(\nn^d,+)$ is an FFM. Fix $x \in M$. It is clear that $\langle x, y \rangle \ge 0$ for all $y \in M$. Thus, $y \mid_M x$ implies that $\norm{y} \le \norm{x}$. As a result, the set $\{a \in \mathcal{A}(M) : a \mid_M x\}$ is finite, and so $\mathsf{Z}(x)$ is also finite. Hence $M$ is an FFM.
\end{proof}

As an immediate consequence of Proposition~\ref{prop:monoids in C are FF monoids}, each monoid in $\mathcal{C}$ satisfies the last four conditions in the chain of implications~(\ref{eq:monoid atomicity taxonomy}).

\smallskip
%%%%%%%%%%%%%%%%%%%%%%%%%
\subsection{Other-Half-Factorial Monoids}

Other-half-factoriality is a dual notion of half-factoriality that was introduced by Coykendall and Smith in~\cite{CS11}.

\begin{definition}
	An atomic monoid $M$ is called an \emph{OHFM} (or an \emph{other-half-factorial monoid}) if for all $x \in M^\bullet$ and $z, z' \in \mathsf{Z}(x)$ the equality $|z| = |z'|$ implies that $z = z'$.
\end{definition}

Although an integral domain is a UFD if and only if its multiplicative monoid is an OHFM~\cite[Corollary~2.11]{CS11}, in general an OHFM is not a UFM (or an HFM), as one can deduce from the next theorem.

A set of points in a $d$-dimensional vector space $V$ over $\rr$ is said to be \emph{affinely independent} provided that for every $k \in \ldb 2,d+1 \rdb$ no $k$ of such points lie in a $(k - 2)$-dimensional affine subspace of $V$. If a set is affinely independent, its points are said to be in \emph{general linear position}.

\begin{theorem} \label{thm:OHF characterizations}
	Let $M$ be a nontrivial monoid in $\mathcal{C}$, and set $V = \rr \otimes_{\zz} \emph{\gp}(M)$. The following statements are equivalent.
	\begin{enumerate}
		\item[(a)] The monoid $M$ is an OHFM.
		\smallskip
		
		\item[(b)] Every face submonoid of $M$ is an OHFM.
		\smallskip
		
		\item[(c)] The points in $\mathcal{A}(M)$ are affinely independent in $V$.
		\smallskip
		
		\item[(d)] The cone $\conv_V(\mathcal{A}(M))$ is a simplex, whose dimension is either $\emph{\rank}(M)-1$ or $\emph{\rank}(M)$.
	\end{enumerate}
\end{theorem}

\begin{proof}
	Set $d := \rank(M)$. By Proposition~\ref{prop:conditions defining monoids in C}, there exists a submonoid $M'$ of $(\nn^d,+)$ such that $M' \cong M$. By tensoring both $\zz$-modules $\zz$ and $\gp(M)$ with $\rr$, one can extend any monoid isomorphism from $M'$ to $M$ to a linear isomorphism from $\rr^d$ to $V$. Since the property of being affinely independent is clearly preserved by isomorphisms, there is no loss of generality in assuming that $M$ is a submonoid of $(\nn^d,+)$, and we do so.
	\smallskip
	
	(a) $\Rightarrow$ (c): Assume that $M$ is an OHFM and suppose, by way of contradiction, that the points in $\mathcal{A}(M)$ are not affinely independent. Then there exist $k \in \ldb 2,d+1 \rdb$ and distinct vectors $a_1, \dots, a_k \in \mathcal{A}(M)$ contained in a $(k-2)$-dimensional affine subspace~$W$ of $V$. Let $M'$ denote now the submonoid of $M$ generated by the atoms $a_1, \dots, a_k$. Since $W - a_k$ is a $(k-2)$-dimensional subspace of $V$, the vectors $a_1 - a_k, \dots, a_{k-1} - a_k$ are linearly dependent in $W - a_k$. This, along with Cramer's rule, guarantees that  $\sum_{i=1}^{k-1} q_i (a_i - a_k) = 0$ for some rational coefficients $q_1, \dots, q_{k-1}$ (not all zeros). After relabeling vectors and coefficients, we can assume the existence of $j \in \ldb 1,k-2 \rdb$ such that $q_1, \dots, q_j$ are negative and $q_{j+1}, \dots, q_{k-1}$ are nonnegative. Set
	\[
		x := \sum_{i=1}^j (-mq_i) a_k + \sum_{i=j+1}^{k-1}(mq_i)a_i \in M',
	\]
	where $m$ is the least common multiple of the denominators of all the nonzero $q_i$'s. Then
	\[
		z := \bigg(\sum_{i=1}^j (-mq_i) \bigg)a_k + \sum_{i=j+1}^{k-1}(mq_i)a_i \quad \text{ and } \quad z' := \quad \sum_{i=1}^j (-mq_i)a_i + \bigg(\sum_{i=j+1}^{k-1} (mq_i) \bigg) a_k
	\]
	are two factorizations in $\mathsf{Z}_{M'}(x)$ having the same length. As $a_1, \dots, a_k$ are also atoms of~$M$, it follows that $z$ and $z'$ are also factorizations in $\mathsf{Z}_M(x)$, which contradicts that~$M$ is an OHFM.
	\smallskip
	
	(c) $\Rightarrow$ (a): Suppose that the points in $\mathcal{A}(M)$ are affinely independent in $V$. We have seen in the proof of Theorem~\ref{thm:factoriality characterizations} that $|\mathcal{A}(M)| \ge \dim \, \cone_V(M)$. So $|\mathcal{A}(M)| \in \{d, d+1\}$. If $|\mathcal{A}(M)| = d$, then Theorem~\ref{thm:factoriality characterizations} ensures that $M$ is a UFM and, therefore, an OHFM. Thus, we assume that $|\mathcal{A}(M)| = d+1$. Let $\mathcal{A}(M) =: \{a_0, a_1, \dots, a_d\}$ and suppose, by way of contradiction, that $M$ is not an OHFM. This guarantees the existence of two distinct nonzero $(d+1)$-tuples $(m_0, \dots, m_d)$ and $(n_0, \dots, n_d)$ in $\nn^d$ satisfying that
	\[
		\sum_{i=0}^d m_i a_i = \sum_{i=0}^d n_i a_i \quad \text{ and } \quad \sum_{i=0}^d (m_i - n_i) = 0.
	\]
	Assume, without loss of generality, that $m_0 \neq n_0$. Let $H$ be an affine hyperplane in $V$ containing $a_1, \dots, a_d$. Take $u \in \rr^d$ and $\alpha \in \rr$ such that $H = \{r \in \rr^d \mid \langle r,u \rangle = \alpha\}$. As the points $a_0, a_1, \dots, a_d$ are affinely independent, $a_0 \notin H$. Then
	\begin{align*}
		0  = \sum_{i=0}^d (m_i - n_i) \langle u, a_i \rangle = (m_0 - n_0) \langle u, a_0 \rangle + \alpha \sum_{i=1}^d (m_i - n_i) = (m_0 - n_0) (\langle u, a_0 \rangle - \alpha).
	\end{align*}
	As a result, $\langle u, a_0 \rangle = \alpha$, which contradicts that $a_0$ does not belong to $H$. Hence~$M$ must be an OHFM.
	\smallskip
	
	(c) $\Rightarrow$ (b): Suppose that the points in $\mathcal{A}(M)$ are affinely independent. Let $N$ be a face submonoid of $M$, and let $F$ be a face of $\cone_V(M)$ satisfying that $N = M \cap F$. Proposition~\ref{prop:atoms of face submonoids} ensures that $\mathcal{A}(N) = \mathcal{A}(M) \cap F$. Since the set of points $\mathcal{A}(M)$ is affinely independent, the set of points $\mathcal{A}(N)$ is also affinely independent. As we have already proved the equivalence of (a) and (c), we can conclude that $N$ is an OHFM. Hence statement~(b) follows.
	\smallskip
	
	(b) $\Rightarrow$ (a): It is clear.
	\smallskip
	
	(c) $\Leftrightarrow$ (d): These two statements are obviously restatements of each other.
\end{proof}

\begin{cor} \label{cor:characterization of OHF NS}
	Let $N$ be a numerical monoid. Then $N$ is an OHFM if and only if the embedding dimension of $N$ is at most $2$.
\end{cor}

\begin{remark}
	Theorem~\ref{thm:OHF characterizations} was indeed motivated by Corollary~\ref{cor:characterization of OHF NS}, which was first proved by Coykendall and Smith in~\cite{CS11}.
\end{remark}

The fact that every proper face submonoid of a monoid $M$ in $\mathcal{C}$ is an OHFM does not guarantee that $M$ is an OHFM, as one can see in the following example.

\begin{example}
	Consider the submonoid $M := \langle 2e_1, 3e_1, 2e_2, 3e_2 \rangle$ of $(\nn^2,+)$. It is easy to argue that $\mathcal{A}(M) = \{2e_1, 3e_1, 2e_2, 3e_2\}$. Observe that the $1$-dimensional faces of $\cone_{\rr^2}(M)$ are $\rr_{\ge 0} e_1$ and $\rr_{\ge 0} e_2$. Then there are two face submonoids of $M$ corresponding to $1$-dimensional faces of $\cone_{\rr^2}(M)$, and they are both isomorphic to the numerical monoid $\langle 2, 3 \rangle$, which is an OHFM by Corollary~\ref{cor:characterization of OHF NS}. Hence every proper face submonoid of $M$ is an OHFM. However, $\conv_{\rr^2}(\mathcal{A}(M))$ is not a simplex and, therefore, it follows from Theorem~\ref{thm:OHF characterizations} that $M$ is not an OHFM. 
\end{example}

We conclude this section with the following proposition.

\begin{prop} \label{prop:non-factorial face submoniods of OHF-monoids}
	Let $M$ be an OHFM in $\mathcal{C}$, and set $V = \rr \otimes_{\zz} \emph{\gp}(M)$. Then the faces of $\cone_V(M)$ whose corresponding face submonoids are not UFMs form a (possibly empty) interval in the face lattice $\mathsf{F}(\cone_V(M))$.
\end{prop}

\begin{proof}
	In light of Proposition~\ref{prop:conditions defining monoids in C} and Proposition~\ref{prop:combinatorial and geometric equivalence of cones}, one can assume that $M$ is a submonoid of $(\nn^d,+)$, where $d := \rank(M)$. In this case, $V = \rr^d$.
	
	Let $\mathcal{N}$ consist of all faces of $\cone_V(M)$ whose corresponding face submonoids are not UFMs. If $M$ is a UFM, it follows from Theorem~\ref{thm:factoriality characterizations} that every face submonoid of $M$ is also a UFM and, therefore, $\mathcal{N}$ is empty. So we assume that $M$ is not a UFM.
	
	Among all the faces in $\mathcal{N}$, let $F$ and $F'$ be minimal in $\mathsf{F}(\cone_V(M))$. Suppose, by way of contradiction, that $F \neq F'$. Set $N := M \cap F$ and $N' := M \cap F'$. It follows from Proposition~\ref{prop:atoms of face submonoids} that $F = \cone_V(N)$ and $F' = \cone_V(N')$. Since $F$ and $F'$ are minimal, they are not comparable and so we can take $a \in \mathcal{A}(N) \setminus \mathcal{A}(N')$ and $a' \in \mathcal{A}(N') \setminus \mathcal{A}(N)$. Once again, one can rely on the minimality of $F$ and $F'$ to obtain
	\[
		\rank (\mathcal{A}(N)) = \rank (\mathcal{A}(N) \setminus \{a\}) \quad \text{ and }  \quad \rank (\mathcal{A}(N')) = \rank (\mathcal{A}(N') \setminus \{a'\}).
	\]
	As a result, the rank of the set $\mathcal{A} := \mathcal{A}(N) \cup \mathcal{A}(N')$ is at most $|\mathcal{A}|-2$. Set $n := |\mathcal{A}|$, and let $A$ be the $d \times n$ real matrix whose columns are the vectors in $\mathcal{A}$ (after some order is fixed). Then $\rk \, A = \rk(\mathcal{A}) \le n-2$. Thus, $\dim \, \ker  A \ge 2$. Consider the hyperplane of $\rr^n$ defined by
	\[
		H := \bigg\{(x_1, \dots, x_n) \in \rr^n \ \bigg{|} \ \sum_{i=1}^n x_i = 0 \bigg\}
	\]
	and notice that
	\[
		\dim (H \cap \ker A) = \dim H + \dim \ker A - \dim \, \text{span}(H \cup \ker \, A) \ge 1.
	\]
	Therefore there is a nonzero vector $(q_1, \dots, q_n) \in \ker A$ satisfying that $\sum_{i=1}^n q_i = 0$. First, taking $j \in \ldb 1,n \rdb$ such that  $q_1, \dots, q_j \le 0$ and $q_{j+1}, \dots, q_n > 0$, then taking~$m$ to be the least common multiple of the denominators of all the nonzero $q_i$'s, and finally proceeding as we did in the second paragraph of the proof of Theorem~\ref{thm:OHF characterizations}, we can obtain two distinct factorizations of the same element of $M$ having the same length. However, this contradicts that $M$ is an OHFM. Hence there exists only one minimal face of $\cone_V(M)$ whose face submonoid is not a UFM, namely, $F$. Clearly, the face submonoid of any face containing $F$ cannot be a UFM. This implies that $[F, \cone_V(M)] \subseteq \mathcal{N}$. The reverse inclusion follows from the uniqueness of a minimal face in $\mathcal{N}$. Hence $\mathcal{N}$ is the interval $[F, \cone_V(M)]$.
\end{proof}

The reverse implication of Proposition~\ref{prop:non-factorial face submoniods of OHF-monoids} does not hold, as the following example illustrates.

\begin{example}
	Consider the submonoid $M := \big\langle 3e_1, 3e_2, 2e_3, 3e_3 \big\rangle$ of $(\nn^3, +)$. It can be readily verified that $\mathcal{A}(M) = \{ 3e_1, 3e_2, 2e_3, 3e_3 \}$. Since $\{2e_3, 3e_3\}$ is an affinely dependent set, it follows from Theorem~\ref{thm:OHF characterizations} that $M$ is not an OHFM. However, the non-UFM face submonoids of $M$ are precisely those determined by the faces of $\cone_V(M)$ contained in the interval $[\rr e_3, \cone_V(M)]$. The face lattice of $\cone_{\rr^3}(M)$ is shown in Figure~\ref{fig:OHF counterexample}.
	\begin{figure}[h]
		\begin{center}
			\includegraphics[width=8cm]{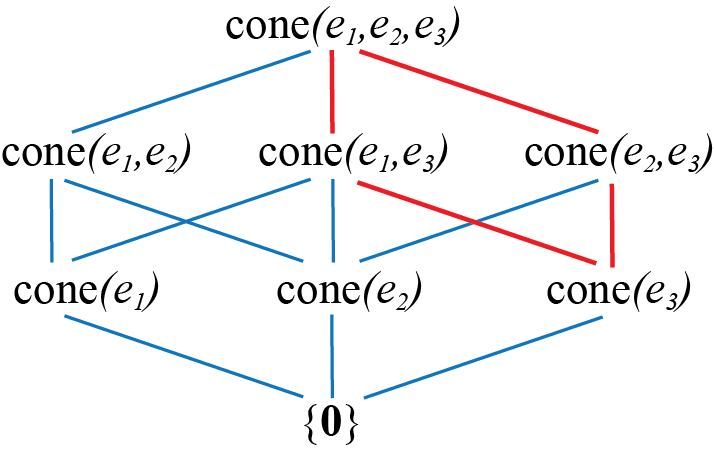}
			\caption{The face lattice of the cone $\cone_{\rr^3}(M)$ with the interval $[\cone(e_3), \cone(e_1, e_2, e_3)]$ consisting of all non-UFM face submonoids of~$M$ (highlighted in red color).}
			\label{fig:OHF counterexample}
		\end{center}
	\end{figure}
\end{example}

\medskip
%%%%%%%%%%%%%%%%%%%%%%%%%%%%%%%%%%
%%%%%%%%%%%%%%%%%%%%%%%%%%%%%%%%%%
\section{Cones of Primary Monoids and Finitary Monoids}
\label{sec:primary monoids and finitary monoids}

 As mentioned at the beginning of this paper, the classes of primary monoids and finitary monoids have been crucial in the development of factorization theory as the arithmetic structure of their members abstracts certain properties of important classes of integral domains. The first part of this section is devoted to investigate some geometric aspects of primary monoids in $\mathcal{C}$. Then we shift our focus to the study of finitary monoids of $\mathcal{C}$.

\smallskip
%%%%%%%%%%%%%%%%%%
\subsection{Primary Monoids}

Recall that a monoid $M$ is primary if $M$ is nontrivial and for all $x,y \in M^\bullet$ there exists $n \in \nn$ such that $ny \in x + M$. The study of primary monoids was initiated by T.~Tamura~\cite{tT74} and M.~Petrich~\cite{mP73} in the 1970s and has received a great deal of attention since then~\cite{fHK95,fHK95a,aG96}. Primary monoids naturally appear in commutative algebra: an integral domain is $1$-dimensional and local if and only if its multiplicative monoid is primary \cite[Theorem~2.1]{aG96}.

The primary monoids in $\mathcal{C}$ are precisely those minimizing the number of face submonoids.

\begin{prop} \label{prop:primary characterization}
	Let $M$ be a nontrivial monoid in $\mathcal{C}$, and set $V = \rr \otimes_{\zz} \emph{\gp}(M)$. The following statements are equivalent.
	\begin{enumerate}
		\item[(a)] The monoid $M$ is a primary monoid.
		\smallskip
		
		\item[(b)] The only face submonoids of $M$ are $\{0\}$ and $M$.
		\smallskip
		
		\item[(c)] The cone $\cone_V(M)^\bullet$ is an open subset of $V$.
	\end{enumerate}
\end{prop}

\begin{proof}
	By using Proposition~\ref{prop:conditions defining monoids in C} and Proposition~\ref{prop:combinatorial and geometric equivalence of cones}, we can assume that $M$ is a submonoid of $(\nn^d,+)$, where $d := \rank(M)$. In this case, $V = \rr^d$.
	\smallskip
	
	(a) $\Leftrightarrow$ (b): It follows from~\cite[Lemma~2.7.7]{GH06b} that $M$ is primary if and only if the only divisor-closed submonoids of $M$ are $\{0\}$ and $M$. This, along with Theorem~\ref{thm:characterization of divisor-closed submonoids in $C$}, implies that the conditions~(a) and~(b) are equivalent.
	\smallskip
	
	(b) $\Rightarrow$ (c): Take $x \in \cone_V(M)^\bullet$. Since $\cone_V(M)$ is the disjoint union of the relative interiors of all its faces, there exists a face $F$ of $\cone_V(M)$ such that $x \in \relin \, F$. As $x \neq 0$, the dimension of $F$ is at least $1$ and, therefore, $M \cap F$ is a nontrivial face submonoid of $M$. It follows now from part~(b) that $M \cap F = M$. As a consequence, $x$ belongs to the relative interior of $\cone_V(M)$. Hence $\cone_V(M)^\bullet$ is open.
	\smallskip
	
	(c) $\Rightarrow$ (b): Since every proper face of $\cone_V(M)$ is contained in the boundary of $\cone_V(M)$, the fact that $\cone_V(M)^\bullet$ is open implies that the only proper face of $\cone_V(M)$ is the origin, from which~(b) follows.
\end{proof}

\begin{remark}
	We want to emphasize that the equivalence (a) $\Leftrightarrow$ (c) in Proposition~\ref{prop:primary characterization} was first established by Geroldinger, Halter-Koch, and Lettl \cite[Theorem~2.4]{GHL95}. However, our approach here is quite different as we have obtained the same result based primarily on the combinatorial structure of the face lattice of $\cone_V(M)$.
\end{remark}

Primary monoids in $\mathcal{C}$ account for all primary submonoids of any (non-necessarily finite-rank) free (commutative) monoid, as the next proposition illustrates.

\begin{prop}
	Let $M$ be a primary submonoid of a free (commutative) monoid. Then~$M$ has finite rank, and $M$ can be embedded into $(\nn^d,+)$, where $d := \emph{\rank}(M)$.
\end{prop}

\begin{proof}
	Let $F_P$ be a free (commutative) monoid on an infinite set $P$ such that $M$ is a submonoid of $F_P$. For $s \in F_P$ and $S \subseteq F_P$, write
	\[
		\text{Spec}(s) := \big\{p \in P \mid p \text{ divides } s \text{ in } F_P \big\} \quad \text{ and } \quad \text{Spec}(S) := \bigcup_{s \in S} \text{Spec}(s).
	\]
	Suppose, by way of contradiction, that $\text{Spec}(M)$ contains infinitely many elements. Fix $x \in M^\bullet$. Since $F_P$ is free and, therefore, a UFM, the set $\text{Spec}(x)$ must be finite. Then we can take $p \in P$ such that $p \in \text{Spec}(M) \setminus \text{Spec}(x)$. Because $p$ is a prime element of~$F_P$, it is not hard to verify that the set $S := \{y \in M \mid p \text{ does not divide } y \text{ in } F_P \}$ is a divisor-closed submonoid of $M$. The fact that $p \notin \text{Spec}(x)$ implies that $S$ is a nonzero submonoid of $M$, and the fact that $p \in \text{Spec}(M)$ implies that $S$ is a proper submonoid of $M$. Therefore the monoid $M$ contains a proper nonzero divisor-closed submonoid, which contradicts that $M$ is primary. As a consequence, $\text{Spec}(M)$ must be finite, and so $M$ can be naturally embedded into the finite-rank free (commutative) monoid $\bigoplus_{i=1}^t \nn p_i$, where $p_1, \dots, p_t$ are the prime elements in $\text{Spec}(M)$. It follows now from Proposition~\ref{prop:conditions defining monoids in C} that $M$ can be embedded into $(\nn^r,+)$.
\end{proof}

\medskip
%%%%%%%%%%%%%%%%%%%%%%
\subsection{Finitely Primary Monoids}
\label{subsec:finitely primary monoids}

Now we restrict our attention to a special subclass of primary monoids that has been key in the development of factorization theory, the class consisting of finitely primary monoids. The initial interest in this class originated from commutative algebra: the multiplicative monoid of a $1$-dimensional local Mori domain with nonempty conductor is finitely primary~\cite[Proposition~2.10.7.6]{GH06b}. Finitely primary monoids were introduced in~\cite{aG96} by Geroldinger.
\smallskip

The \emph{complete integral closure} of a monoid $M$, denoted by $\widehat{M}$, is defined as follows:
\[
	\widehat{M} := \big\{x \in \text{gp}(M) \mid \text{ there exists } y \in M \text{ such that } nx + y \in M \text{ for every } n \in \nn \big\}. 
\]
Clearly, $\widehat{M}$ is a submonoid of $\text{gp}(M)$ containing $M$, and so $\rank(\widehat{M}) = \rank(M)$. A monoid $M$ is called \emph{finitely primary} if there exist $d \in \nn$ and a UFM $F := \langle p_1, \dots, p_d \rangle$, where $p_1, \dots, p_d$ are pairwise distinct prime elements in $F$, such that the following conditions hold:
\begin{enumerate}
	\item the monoid $M$ is a submonoid of $F$,
	\smallskip
	
	\item the inclusion $M^\bullet \subseteq p_1 + \dots + p_d + F$ holds, and
	\smallskip
	
	\item the inclusion $\alpha(p_1 + \dots + p_d) + F \subseteq M$ holds for some $\alpha \in \nn^\bullet$.
\end{enumerate}
In this case, it follows from \cite[Theorem~2.9.2]{GH06b} that $\widehat{M} \cong (\nn^d,+)$. Then $\rk(M) = d$ and, moreover, any finitely primary monoid of rank $d$ belongs to $\mathcal{C}_d$. On the other hand, it also follows from~\cite[Theorem~2.9.2]{GH06b} that finitely primary monoids are primary. Therefore Proposition~\ref{prop:primary characterization} guarantees that for any finitely primary monoid $M$ the set $\cone_V(M)^\bullet$ is open in the finite-dimensional $\rr$-vector space $V := \rr \otimes_{\zz} \text{\gp}(M)$. This implies, in particular, that nontrivial finitely primary monoids cannot be finitely generated. As the following theorem reveals, the closures of their cones happen to be simplicial cones.

\begin{theorem} \label{thm:the closure of the cone of a finitely primary monoid is rational and simplicial}
	Let $M$ be a finitely primary monoid, and set $V = \rr \otimes_{\zz} \emph{\gp}(M)$. Then~$M$ belongs to $\mathcal{C}$, and the set $\overline{\cone_V(M)}$ is a rational simplicial cone in $V$.
\end{theorem}

\begin{proof}
	Let $d$ be the rank of $M$. We have already observed that $M$ is in the class $\mathcal{C}$. Then by virtue of Proposition~\ref{prop:conditions defining monoids in C} and Proposition~\ref{prop:combinatorial and geometric equivalence of cones}, we can assume that $M$ is a submonoid of $(\nn^d,+)$ and, therefore, that $V = \rr^d$. In addition, there is no loss in assuming that $M \subseteq \widehat{M} \subseteq \nn^d$.
	
	Because $\widehat{M}\cong (\nn^d,+)$, one can take distinct prime elements $p_1, \dots, p_d$ of $\widehat{M}$ such that $\widehat{M} = \langle p_1, \dots, p_d \rangle = \bigoplus_{i=1}^d \nn p_i$. It follows from \cite[Theorem~2.9.2]{GH06b} that 
	\[
		M^\bullet \subseteq p_1 + \dots + p_d + \widehat{M} \ \text{ and } \ \alpha(p_1 + \dots + p_d) + \widehat{M} \subseteq M
	\]
	for some $\alpha \in \nn^\bullet$. Let $C_p$ be the cone in $V$ generated by $p_1,\dots, p_d$. Clearly, $C_p$ is a rational simplicial cone of dimension~$d$. We claim that $\overline{\cone_V(M)} = C_p$. Since
	\[
		M^\bullet \subseteq p_1 + \dots + p_d + \widehat{M} \subset \inter \, C_p,
	\]
	the inclusion $M \subseteq C_p$ holds. As a consequence, $\cone_V(M) \subseteq C_p$ and, as the cone $C_p$ is a closed set in $V$, the inclusion $\overline{\cone_V(M)} \subseteq C_p$ follows. Let us proceed to argue that $C_p \subseteq \overline{\cone_V(M)}$. To do so, fix $\epsilon > 0$ and fix also an index $j \in \ldb 1,d \rdb$. Let $L$ be the $1$-dimensional face of $C_p$ in the direction of the vector $p_j$, and consider the conical open ball with central axis $L$ given by
	\[
		B(p_j, \epsilon) := \bigg\{ w \in V \setminus \{0\} \ \bigg{|} \ \frac{\norm{w - \mathsf{p}_L(w)}}{ \norm{w}} < \epsilon \bigg\},
	\]
	where $\mathsf{p}_L \colon V \to \rr p_j$ is the linear projection of $V$ onto its subspace $\rr p_j$. It is clear that the set $\{0\} \cup \big( B(p_j, \epsilon) \cap \inter \, C_p \big)$ is a $d$-dimensional subcone of $C_p$ and, therefore, it must intersect $\widehat{M}$. Then one can take $y \in \widehat{M} \cap \inter \, C_p$ such that $\rr_{> 0} y \subset B(p_j, \epsilon)$. Because
	\[
		\alpha \big( \widehat{M} \cap \inter \, C_p \big) \subseteq \alpha(p_1 + \dots + p_d) + \widehat{M} \subseteq M,
	\]
	$\alpha y \in M$. As a result, $\rr_{> 0} y \subset \cone_V(M)$. As $\cone_V(M)$ and every open conical ball with central axis $L$ have an open ray in common, $p_j \in L \subseteq \overline{\cone_V(M)}$. Because the index~$j$ was arbitrarily taken, $p_j \in \overline{\cone_V(M)}$ for every $j \in \ldb 1,d \rdb$, and so $C_p \subseteq \overline{\cone_V(M)}$. Hence $\overline{\cone_V(M)}$ is a rational simplicial cone.
\end{proof}

With notation as in Theorem~\ref{thm:the closure of the cone of a finitely primary monoid is rational and simplicial}, the facts that $M$ is primary and $\overline{\cone_V(M)}$ is a rational simplicial cone do not guarantee that $M$ is finitely primary. The following example sheds some light upon this observation.

\begin{example}
	Consider the subset $M$ of $\nn^2$ defined by
	\[
		M := \{(0,0)\} \cup \big\{ (n,m) \in \nn^2 \mid n,m \in \nn^\bullet \text{ and } m \le 2^n \big\}.
	\]
	From the fact that $f(x) = 2^x$ is a convex function, one can readily verify that $M$ is a submonoid of $(\nn^2,+)$. Since $M$ contains $(n,1)$ for every $n \in \nn^\bullet$, the ray $\rr_{\ge 0} e_1$ is contained in $\overline{\cone_{\rr^2}(M)}$. On the other hand, the fact that $\{(n, 2^n) \mid n \in \nn^\bullet \} \subset M$, along with $\lim_{n \to \infty} 2^n/n = \infty$, guarantees that the ray $\rr_{\ge 0} e_2$ is contained in $\overline{\cone_{\rr^2}(M)}$. Thus, $\overline{\cone_{\rr^2}(M)}$ is the closure of the first quadrant and so
	\[
		\cone_{\rr^2}(M) = \big\{ (0,0)\} \cup \{(x,y) \in \rr^2 \mid x,y > 0 \big\} = \{(0,0)\} \cup \rr^2_{> 0}.
	\]
	Because $\cone_{\rr^2}(M)^\bullet$ is an open set in $\rr^2$, it follows from Proposition~\ref{prop:primary characterization} that $M$ is a primary monoid. On the other hand, the equality $\overline{\cone_{\rr^2}(M)} = \rr^2_{\ge 0}$ holds, and so $\overline{\cone_{\rr^2}(M)}$ is a rational simplicial cone.
	
	To argue that $M$ is not finitely primary, it suffices to verify that $\widehat{M} \not\cong (\nn^2,+)$. To do so, fix $m \in \nn$, and then take $N \in \nn$ large enough so that $nm \le 2^n$ for every $n \ge N$. Note that $y := (N,Nm)$ belongs to $M$. Moreover,
	\[
		n(1,m) + y = (n+N, (n+N)m) \in M
	\]
	for every $n \in \nn$. Therefore $(1,m) \in \widehat{M}$ for every $m \in \nn$. On the other hand, for any $m \in \nn^\bullet$ and $(a,b) \in M^\bullet$,
	\[
		2^a(0,m) + (a,b) = (a, 2^am + b) \notin M.
	\]
	Hence $(n,m) \in \widehat{M}^\bullet$ implies that $n > 0$. As a result, $\widehat{M} = \{(n,m) \in \nn^2 \mid n > 0\}$. Since $\mathcal{A}(\widehat{M}) = \{(1,n) \mid n \in \nn\}$ contains infinitely many elements, $\widehat{M} \not\cong (\nn^2,+)$. Consequently,~$M$ cannot be finitely primary.
\end{example}

\smallskip
%%%%%%%%%%%%%%%%%%
\subsection{Finitary Monoids}

Let $M$ be a monoid. Recall that $M$ is finitary if it is a BFM and there exist a finite subset $S$ of $M$ and a positive integer $n$ satisfying that $n M^\bullet \subseteq S + M$. Geroldinger et al. introduced and studied the class of finitary monoids in~\cite{GHHK03} motivated by the fact that monoids in this class naturally show in commutative ring theory: the multiplicative monoid of a Noetherian domain $R$ is finitary if and only if~$R$ is $1$-dimensional and semilocal~\cite[Proposition~4.14]{GHHK03}. In addition, finitary monoids conveniently capture certain aspects of the arithmetic structure of more sophisticated monoids, including $v$-Noetherian $G$-monoids~\cite{aG93} and congruence monoids~\cite{GH04}. Every finitely generated monoid is finitary. In particular, affine monoids are finitary.

The monoid $M$ is said to be \emph{weakly finitary} if there exist a finite subset $S$ of $M$ and $n \in \nn^\bullet$ such that $nx \in S + M$ for all $x \in M^\bullet$. Clearly, every finitary monoid is weakly finitary. The face submonoids of a monoid in $\mathcal{C}$ inherit the condition of being (weakly) finitary.

\begin{prop} \label{prop:face-submonoids inherit finitariness}
	Let $M$ be a monoid in $\mathcal{C}$. Then $M$ is finitary (resp., weakly finitary) if and only if each face submonoid of $M$ is finitary (resp., weakly finitary).
\end{prop}

\begin{proof}
	We will prove only the finitary version of the proposition as the weakly finitary version follows similarly. By Proposition~\ref{prop:conditions defining monoids in C} we can assume that $M$ is a submonoid of $(\nn^d,+)$, where $d$ is the rank of $M$.
	
	Suppose that $M$ is finitary. Take~$F$ to be a face of $\cone_{\rr^d}(M)$, and consider the face submonoid $N := M \cap F$. Since $M$ is finitary, there exist $n \in \nn$ and a finite subset $S$ of~$M$ such that $n M^\bullet \subseteq S + M$. We claim that $n N^\bullet \subseteq S_F + N$, where $S_F := S \cap F$. Take $x_1, \dots, x_n \in N^\bullet = M^\bullet \cap F$. Because $n (M^\bullet \cap F) \subseteq n M^\bullet \subseteq S + M$, there exist $s \in S$ and $y \in M$ such that $\sum_{i=1}^n x_i = s + y$. Since $M \cap F$ is a divisor-closed submonoid of $M$, we find that $s,y \in F$. Therefore $s \in S_F$ and $y \in N$; this implies that $\sum_{i=1}^n x_i \in S_F + N$. Hence $N$ is a finitary monoid. The reverse implication follows trivially as $\cone_{\rr^d}(M)$ is a face of itself.
\end{proof}

Our next goal is to give a sufficient geometric condition for a monoid in $\mathcal{C}$ to be finitary. First, let us recall the notion of triangulation. Let $V$ be a finite-dimensional vector space over an ordered field $\ff$. A \emph{conical polyhedral complex} $\pp$ in $V$ is a collection of polyhedral cones in $V$ satisfying the following conditions:
\begin{enumerate}
	\item every face of a polyhedron in $\pp$ is also in $\pp$, and
	\smallskip
	
	\item the intersection of any two polyhedral cones $C_1$ and $C_2$ in $\pp$ is a face of both $C_1$ and~$C_2$.
\end{enumerate}
Clearly, the underlying set of the face lattice of a given polyhedral cone is a conical polyhedral complex. For a conical polyhedral complex $\pp$ in $V$, we set $|\pp| := \bigcup_{C \in \pp} \, C$. Let $\pp$ and $\pp'$ be two conical polyhedral complexes. The complex $\pp'$ is said to be a \emph{polyhedral subdivision} of $\pp$ provided that $|\pp| = |\pp'|$ and each face of $\pp$ is the union of faces of $\pp'$. A polyhedral subdivision $\mathsf{T}$ of~$\pp$ is called a \emph{triangulation} of $\pp$ provided that $\mathsf{T}$ consists of simplicial cones. Every conical polyhedral complex has the following special triangulations.

\begin{theorem}\cite[Theorem~1.54]{BG09} \label{thm:exsitence of triangulations}
	Let $\pp$ be a conical polyhedral complex in $\rr^d$ for some $d \in \nn^\bullet$, and let $S \subset |\pp|$ be a finite set of nonzero vectors such that $S \cap C$ generates $C$ for each $C \in \pp$. Then there exists a triangulation $\mathsf{T}$ of $\pp$ such that $\{\rr_{\ge 0}v \mid v \in S\}$ is the set  of $1$-dimensional faces of $\mathsf{T}$.
\end{theorem}

We are in a position now to offer a sufficient geometric condition for a monoid in $\mathcal{C}$ to be finitary.

\begin{theorem} \label{thm:finitary GAMs}
	Let $M$ be a monoid in $\mathcal{C}$, and set $V = \rr \otimes_{\zz} \emph{\gp}(M)$. If $\cone_V(M)$ is polyhedral, then $M$ is finitary.
\end{theorem}

\begin{proof}
	Let $d$ be the rank of $M$. Based on Proposition~\ref{prop:conditions defining monoids in C} and Proposition~\ref{prop:combinatorial and geometric equivalence of cones}, one can assume that $M \subseteq \nn^d$. In this case, $V = \rr^d$. Since $\cone_V(M)$ is polyhedral, it follows by Farkas-Minkowski-Weyl Theorem that $\cone_V(M)$ is the conic hull of a finite set of vectors. As the vectors in such a generating set are nonnegative rational linear combinations of vectors in $M$, there exists $S = \{v_1, \dots, v_k\} \subset M$ with $k \ge d$ such that $\cone_V(M) = \cone_V(S)$. By Theorem~\ref{thm:exsitence of triangulations}, there exists a triangulation $\mathsf{T}$ of the face lattice of $\cone_V(M)$ whose set of $1$-dimensional faces is $\{\rr_{\ge 0} v_i \mid i \in \ldb 1,k \rdb \}$. Then for any $T \in \mathsf{T}$ there are unique indices $t_1, \dots, t_d$ satisfying that
	\[
		1 \le t_1 < \dots < t_d \le k \quad \text{ and } \quad T = \cone_V(v_{t_1}, \dots, v_{t_d}),
	\]
	and we can use this to assign to $T$ the parallelepiped
	\[
		\Pi_T := \bigg\{ \sum_{i=1}^d \alpha_1 v_{t_i} \ \bigg{|} \ 0 \le \alpha_i < 1 \ \text{for every} \ i \in \ldb 1,d \rdb \bigg\}.
	\]
	It is clear that
	\[
		|\Pi_T \cap \zz^d| < \infty \quad \text{ and } \quad \Pi_T \cap \zz^d \subset \sum_{i=1}^d \qq_{\ge 0} v_{t_i}.
	\]
	Then we can choose $N_T \in \nn$ sufficiently large so that $N_T v \in \sum_{i=1}^d \nn v_{t_i}$ for every $v \in \Pi_T \cap \zz^d$. Now take $m := \max \{N_T \, |\Pi_T \cap \zz^d| : T \in \mathsf{T}\}$ and set $n := m \, |\mathcal{T}|$. In order to show that $M$ is finitary, it suffices to verify that $nM^\bullet \subseteq S + M$.
	
	Take (possibly repeated) elements $x_1, \dots, x_n \in M^\bullet$. For every $x \in \{x_1,\dots, x_n\}$, there exists $T \in \mathsf{T}$ with $x \in T$. Let $T = \cone_V(v_{t_1}, \dots, v_{t_d})$ for $t_1 < \dots < t_d$ be a simplicial cone in~$\mathsf{T}$. Observe that we can naturally partition $T$ into (translated) copies of the parallelepiped $\Pi_T$, that is, $T$ equals the disjoint union of the sets $v + \Pi_T$ for every $v \in \nn \sum_{i=1}^d v_{t_i}$. As a result, there exist $z \in \Pi_T \cap \zz^d$ and coefficients $\alpha_1, \dots, \alpha_d \in \nn$ satisfying that
	\begin{equation} \label{eq:parallelepiped decomposition}
		x = z + \sum_{i=1}^d \alpha_i v_{t_i}.
	\end{equation}
	Hence for $i \in \ldb 1,n \rdb$, we can write $x_i = z_i + m_i$ for some $z_i \in \bigcup_{T \in \mathsf{T}} \Pi_T \cap \zz^d$ and $m_i \in M$. Since $n = m |\mathsf{T}|$, there exists $T_0 \in \mathsf{T}$ such that
	\[
		|\{i \in \ldb 1,n \rdb \mid z_i \in \Pi_{T_0} \cap \zz^d\}| \ge m.
	\]
	Consider now the equivalence relation on the set of indices $\{i \in \ldb 1,n \rdb \mid z_i \in T_0\}$ defined by $i \sim j$ whenever $z_i = z_j$. The fact that $m \ge N_{T_0} \, |\Pi_{T_0} \cap \zz^d|$ guarantees the existence of a class~$I$ determined by the relation $\sim$ and containing at least $N_{T_0}$ distinct indices. Take $I_0 \subseteq I$ such that $|I_0| = N_{T_0}$. Setting $z := z_i$ for some $i \in I_0$, we see that
	\[
		\sum_{i \in I_0} z_i = N_{T_0}z \in \nn v_1 + \dots + \nn v_n \in S + M
	\]
	and, therefore, $\sum_{i \in I_0} z_i = v + m$ for some $v \in S$ and $m \in M$. As a result, one can set $m' := \sum_{i=1}^n x_i - \sum_{i \in I_0} x_i \in M$ to obtain that
	\[
		\sum_{i=1}^n x_i = \bigg( \sum_{i \in I_0} x_i \bigg) + m' = \sum_{i \in I_0} z_i + m' + \sum_{i \in I_0} m_i = v + \bigg(m + m' + \sum_{i \in I_0} m_i \bigg) \in S + M.
	\]
	Since the elements $x_1, \dots, x_n$ were arbitrarily taken in $M^\bullet$, the inclusion $n M^\bullet \subseteq S + M$ holds. Hence the monoid $M$ is finitary.
\end{proof}

According to the characterization of cones generated by monoids in $\mathcal{C}$ we have provided in Theorem~\ref{thm:a characterization of cones generated by monoids in C}, every $d$-dimensional positive cone $C$ of $\rr^d$ with $C^\bullet$ open can be generated by a monoid in $\mathcal{C}$. Indeed, any such a cone can be generated by a finitary monoid in $\mathcal{C}$.

\begin{prop} \label{prop:a characterization of cones generated by monoids in C}
	For $d \in \nn$, let $C$ be a positive cone in $\rr^d$. If $C^\bullet$ is open in $\rr^d$, then~$C$ can be generated by a finitary monoid in $\mathcal{C}$.
\end{prop}

\begin{proof}
	Assume that $C^\bullet$ is open in $\rr^d$. Take $M = \nn^d \cap C$. It is clear that $C = \cone_{\rr^d}(M)$. Now take $v_0 \in M^\bullet$, and consider the monoid $M' := \{0\} \cup (v_0 + M)$. Let $C'$ be the cone generated by $M'$. Notice that $\qq^d \cap C$ and $\qq^d \cap C'$ are the cones generated by $M$ and $M'$ over $\qq$, respectively. So proving that $C' = C$ amounts to showing that $\qq^d \cap C' = \qq^d \cap C$ (see~\cite[Proposition~1.70]{BG09}). Since $M' \subseteq M$ it follows that $\qq^d \cap C' \subseteq \qq^d \cap C$. Now let $\ell_0$ be the distance from $\{v_0\}$ to $\rr^d \setminus C$. As $\rr^d \setminus C$ is closed and $\{v_0\}$ is compact, $\ell_0 > 0$. Now take $v \in \qq^d \cap C^\bullet$ such that $\norm{v} > 1$, and let $\ell$ be the distance from $v$ to $\rr^d \setminus C$. By a similar argument, $\ell > 0$. Notice that the conical ball
	\[
		B(v, \ell) := \bigg\{ w \in \qq^d \ \bigg{|} \ \frac{\norm{w - \mathsf{p}_v(w)}}{ \norm{w}} < \frac{\ell}{2} \bigg\}
	\]
	is contained in $\qq^d \cap C$; here $\mathsf{p}_v \colon \rr^d \to \rr v$ is the projection of $\rr^d$ onto its one-dimensional subspace $\rr v$. Take $N \in \nn$ such that
	\[
		N > \max \bigg\{ \frac{\norm{v_0}}{ \norm{v} - 1}, \frac{2 \norm{v_0}}{\ell} \bigg\}
	\]
	and $Nv \in \qq^d$. Now set $w_0 := Nv - v_0$. Notice that $\norm{w_0} \ge N \norm{v} - \norm{v_0} > N$. Then
	\[
		\frac{ \norm{w_0 - \mathsf{p}_v(w_0)}}{ \norm{w_0}} < \frac{ \norm{v_0 - \mathsf{p}_v(v_0) }}{N} \le \frac{ \norm{v_0}}{N} < \frac{\ell}{2}.
	\]
	Hence $w_0 \in \qq^d \cap B(v,\ell) \subseteq \qq^d \cap C$, and so there exist coefficients $c_1, \dots, c_k \in \qq_{> 0}$ and elements $v_1, \dots, v_k \in M^\bullet$ such that $w_0 = \sum_{i=1}^k c_i v_i$. As a consequence, one has that $nv \in M' \subseteq \cone_V(M')$ for some $n \in \nn$ and, therefore, $v \in \qq^d \cap \cone_V(M')$. Hence $\qq^d \cap C^\bullet \subseteq \qq^d \cap C$.
	
	As $M'$ generates $C$, we only need to verify that $M'$ is finitary. Take $w_1, w_2 \in M'^\bullet$, and then $v_1, v_2 \in M$ such that $w_1 = v_0 + v_1$ and $w_2 = v_0 + v_2$. Then
	\[
		w_1 + w_2 = v_0 + (v_0 + v_1 + v_2) \in v_0 + (v_0 + M) \subset v_0 + M'.
	\]
	As a result, $2M'^\bullet \subseteq v_0 + M'$, which implies that $M'$ is a finitary monoid.
\end{proof}

Theorem~\ref{thm:finitary GAMs} and Proposition~\ref{prop:a characterization of cones generated by monoids in C} indicate that there is a huge variety of finitary monoids in~$\mathcal{C}$. We proceed to exhibit a monoid in $\mathcal{C}_2$ that is not even weakly finitary. First, let us introduce the following notation.
\medskip

\noindent \textbf{Notation:} For $x \in \rr^2_{\ge 0} \setminus \{(0,0)\}$, we let $\slp(x) \in \rr_{\ge 0} \cup \{\infty\}$ denote the slope of the line $\rr x$, and for $X \subset \rr_{\ge 0}^2$ we set
\[
	\slp(X) := \{\slp(x) \mid x \in X^\bullet\}.
\]

\begin{example}
	Construct a sequence $(v_n)_{n \ge 1}$ of vectors in $\nn^\bullet \times \nn^\bullet$ as follows. Set $v_1 =~(1,1)$ and suppose that for some $n \in \nn$  and for every $i \in \ldb 1,n-1 \rdb$ we have chosen vectors $v_i = (x_i, y_i) \in \nn^\bullet \times \nn^\bullet$ satisfying that $\slp(v_i) < \slp(v_{i+1})$ and $i \norm{v_i} < \norm{v_{i+1}}$. Then take $v_{n+1} = (x_{n+1}, y_{n+1}) \in \nn^2$ in such a way that $x_{n+1} > 0$, $\slp(v_{n+1}) > \slp(v_n)$, and $\norm{v_{n+1}} > n \norm{v_n}$. Now one can consider the additive submonoid $M := \langle v_n \mid n \in \nn^\bullet \rangle$ of $\nn^2$. Clearly, $\mathcal{A}(M) \subseteq \{v_n \mid n \in \nn^\bullet\}$. On the other hand, the fact that $\norm{v_m} > \norm{v_n}$ when $m > n$ implies that only atoms in $\{v_1, \dots, v_{n-1}\}$ can divide $v_n$ in $M$. This, along with the fact that
	\[
		\slp(v_n) > \max \big\{ \slp(v_i) \mid i \in \ldb 1,n-1 \rdb \big\}
	\]
	for every $n \in \nn$, guarantees that $\mathcal{A}(M) = \{v_n \mid n \in \nn^\bullet\}$. Finally, let us verify that $M$ is not weakly finitary. Assume for a contradiction that there exist $n \in \nn$ and a finite subset $S$ of $M$ such that $nx \in S + M$ for all $x \in M^\bullet$. We can assume without loss of generality that $S \subseteq \mathcal{A}(M)$, so we let $S = \{v_{n_1}, \dots, v_{n_k}\}$, where $n_1 < \dots < n_k$. Take $N > \max \{n, n_k\}$. Then write $n' v_N = v_{n_i} + m$ for some $i \in \ldb 1,k \rdb$ and $m \in M$ such that $n' \le n$ and $v_N \nmid_M m$. Since $\slp (n' v_N) > \slp(v_{n_i})$, there exists $j > N$ such that $v_j \mid_M m$. Therefore
	\[
		\norm{Nv_N} > \norm{n'v_N} = \norm{ v_{n_i} + m } > \norm{m} \ge \norm{v_j} \ge \norm{v_{N+1}},
	\]
	which is a contradiction. Hence $M$ is not weakly finitary.
\end{example}

\smallskip
%%%%%%%%%%%%%%%%%%%%%%%
\subsection{Strongly Primary Monoids}

We conclude this paper with a few words about strongly primary monoids in $\mathcal{C}$. Recall that a monoid is strongly primary if it is simultaneously finitary and primary. Numerical monoids and $v$-Noetherian primary monoids are strongly primary. On the other hand, the multiplicative monoid of a $1$-dimensional local Mori domain is strongly primary. Finally, the class of strongly primary monoids also contains the class of finitely primary monoids~\cite[Theorem~2.9.2]{GH06b}. Strongly primary monoids have been investigated in~\cite{GGT19,GHL07,GR18}.

Let $M$ be a monoid. For $x \in M^\bullet$ the smallest $n \in \nn$ satisfying that $nM^\bullet \subseteq x + M$ is denoted by $\mathcal{M}(x)$. When such $n$ does not exist, we set $\mathcal{M}(x) = \infty$. If $M$ is strongly primary, then $\mathcal{M}(x) < \infty$ for all $x \in M^\bullet$~\cite[Lemma~2.7.7]{GH06b}. In addition, set
\[
	\mathcal{M}(M) := \sup \{ \mathcal{M}(a) \mid a \in \mathcal{A}(M) \} \subseteq \nn \cup \{\infty\}.
\]

\begin{example} \label{ex:strongly primary monoid with infinite M(M)}
	Consider the monoid
	\[
		M := \{(0,0)\} \cup \{(x,y) \in \nn^2 \mid x,y > 0\}.
	\]
	It is clear that
	\[
		A := \{(a,b) \in M \mid a=1 \ \text{ or } \ b=1 \} \subseteq \mathcal{A}(M).
	\]
	On the other hand, if $(x,y) \in M^\bullet \setminus A$, then $x,y \ge 2$ and, therefore,
	\[
		(x,y) = (1,1) + (x-1,y-1) \in M^\bullet + M^\bullet.
	\]
	Hence $\mathcal{A}(M) = A$. In addition, the fact that $(1,1) \mid_M (x,y)$ for all $(x,y) \in M^\bullet \setminus A$ implies that $\mathcal{M}((1,1)) = 2$. The inclusion $2M^\bullet \subseteq (1,1) + M$ implies that $M$ is a finitary monoid. On the other hand, $\cone_{\rr^2}(M)^\bullet$ is the open first quadrant, which implies via Proposition~\ref{prop:primary characterization} that $M$ is a primary monoid. As a result, $M$ is strongly primary. Now fix $n \in \zz_{\ge 2}$. Note that if $(n,1) \mid_M m(1,1)$ for some $m \in \nn$, then $m \ge n+1$. Thus, $\mathcal{M}((n,1)) \ge n+1$. On the other hand, if $(x,y) \in (n+1) M^\bullet$, then $x \ge n+1$ and $y \ge 2$, which implies that $(x,y) - (n,1) \in M$. As a result, $\mathcal{M}((n,1)) = n+1$ and, by a similar argument, $\mathcal{M}((1,n)) = n+1$. Hence $\mathcal{M}((a,b)) = a+b$ for every $(a,b) \in A$ and, in particular, $\mathcal{M}(M) = \infty$.
\end{example}

Unlike the computations shown in Example~\ref{ex:strongly primary monoid with infinite M(M)}, an explicit computation of the set $\{\mathcal{M}(a) \mid a \in \mathcal{M}\}$ for a monoid $M$ in $\mathcal{C}$ can be hard to carry out. However, for most monoids $M$ in $\mathcal{C}$ one can argue that $\mathcal{M}(M) = \infty$ without performing such computations.

\begin{prop}
	Let $M$ be a strongly primary monoid in $\mathcal{C}$, and set $V = \rr \otimes_{\zz} \emph{\gp}(M)$. The following statements are equivalent.
	\begin{enumerate}
		\item[(a)] The inequality $\mathcal{M}(M) < \infty$ holds.
		\smallskip
		
		\item[(b)] The equality $\dim \cone_V(M) = 1$ holds.
		\smallskip
		
		\item[(c)] The monoid $M$ is isomorphic to a numerical monoid.
	\end{enumerate}
\end{prop}

\begin{proof}
	As a result of Proposition~\ref{prop:conditions defining monoids in C} and Proposition~\ref{prop:combinatorial and geometric equivalence of cones}, one can assume that $M$ is a submonoid of $(\nn^d,+)$, where $d$ is the rank of $M$. Then $V = \rr^d$.
	\smallskip
	
	(b) $\Leftrightarrow$ (c): It is clear.
	\smallskip
	
	(a) $\Rightarrow$ (b): To argue this suppose, by way of contradiction, that $\dim \cone_V(M) \neq 1$. Since $M$ is strongly primary $M^\bullet$ is not empty and, thus, $\dim \cone_V(M) \ge~2$. As $M$ is primary, $\cone_V(M)^\bullet$ is an open set of $V$ by Proposition~\ref{prop:primary characterization}. Therefore $M$ cannot be finitely generated, which means that $|\mathcal{A}(M)| = \infty$. Since $\{a \in \mathcal{A}(M) \mid \norm{a} < n \}$ is a finite set for every $n \in \nn$, there exists a sequence $(a_n)_{n \in \nn}$ of atoms of $M$ satisfying that $\lim_{n \to \infty} \norm{a_n} = \infty$. Now fix $x \in M^\bullet$. Because $\mathcal{M}(a_n)x = a_n + b$ for some $b \in M$,
	\[
		\lim_{n \to \infty} \mathcal{M}(a_n) = \lim_{n \to \infty} \frac{ \norm{a_n + b}}{ \norm{x}} \ge \frac{1}{ \norm{x}}\lim_{n \to \infty} \norm{a_n} = \infty.
	\]
	Hence $\mathcal{M}(M) = \infty$, which is a contradiction.
	\smallskip
	
	(b) $\Rightarrow$ (a): To argue this, suppose that $\dim \cone_V(M) = 1$. In this case, $M$ is isomorphic to a numerical monoid. Since numerical monoids are finitely generated, the inequality $\mathcal{M}(M) < \infty$ holds.
\end{proof}

\bigskip
%%%%%%%%%%%%%%%%%
\section*{Acknowledgments}

While working on this paper, the author was supported by the NSF~AGEP and the UC Year Dissertation Fellowship. The author would like to thank an anonymous referee for her/his useful comments.

\bigskip
%%%%%%%%%%%%%%%%

\end{document}